\documentclass[11pt,a4paper]{amsart}
\usepackage[T1]{fontenc}
\usepackage{lmodern}
\usepackage[margin=27mm]{geometry}
\usepackage{amsmath,amssymb,amsthm,mathtools,slashed,mathrsfs}
\usepackage{microtype}
\usepackage[colorlinks=true,linkcolor=blue,citecolor=blue,urlcolor=blue]{hyperref}
\usepackage[nameinlink,noabbrev,capitalise]{cleveref}
\numberwithin{equation}{section}
\newtheorem{theorem}{Theorem}[section]
\newtheorem{proposition}[theorem]{Proposition}
\newtheorem{lemma}[theorem]{Lemma}
\newtheorem{corollary}[theorem]{Corollary}
\theoremstyle{remark}

\newtheorem{example}[theorem]{Example}
\newcommand{\RR}{\mathbb R}
\newcommand{\CC}{\mathbb C}
\newcommand{\ZZ}{\mathbb Z}
\newcommand{\NN}{\mathbb N}
\newcommand{\ii}{\mathrm i}
\newcommand{\D}{\slashed D}
\newcommand{\cR}{\mathcal R}
\newcommand{\Rea}{\operatorname{Re}}
\newcommand{\id}{\operatorname{Id}}
\newcommand{\rank}{\operatorname{rank}}
\newcommand{\Area}{\operatorname{Area}}
\newcommand{\inj}{\operatorname{inj}}
\newcommand{\tnabla}{\widetilde\nabla}
\newcommand{\pair}[2]{\langle #1,#2\rangle}
\newcommand{\vol}{\operatorname{vol}}

\newcommand{\dvol}{\,d\operatorname{vol}}
\allowdisplaybreaks[1]

\title[Coupled Dirac-harmonic maps from tori]
{Coupled Dirac-harmonic maps from closed spin tori}
\author[J\"urgen Jost]{J\"urgen Jost}
\address{Max Planck Institute for Mathematics in the Sciences,
Inselstrasse 22, 04103 Leipzig, Germany}
\email{jost@mis.mpg.de}
\author[Jingyong Zhu]{Jingyong Zhu}
\address{Department of Mathematics, Sichuan University,
Chengdu 610065, China}
\email{jzhu@scu.edu.cn}
\date{September 26, 2026}

\begin{document}
\begin{abstract}
We construct Dirac-harmonic maps from closed spin tori whose map
components are not harmonic. For a closed Riemannian surface as
target, we prove that such solutions exist for every source metric
and every spin structure whenever the target metric is nonflat.
The maps can be chosen null-homotopic,
of rank one everywhere, with nowhere-vanishing tension and
unbounded energy. For a strictly negatively curved closed oriented
target surface, every homotopy class contains such a sequence
when the source has the trivial spin structure. For the other
spin structures we give a compatibility condition for the same
construction. Adding a harmonic circle component yields coupled
immersions into the product of any such surface with a circle.
The construction combines a spinor phase that
cancels the rotation of a moving target frame with a curvature
balance along a closed curve. Minimizing the energy of short loops
with prescribed signed curvature area produces the required curves
without any symmetry assumption on the target. Their total
geodesic curvature determines the spinor's periodicity. We also
give a curve criterion in arbitrary target dimension, applications
to totally geodesic surfaces and warped products, and a local
choice of target metric that works on every closed manifold of
dimension at least two.
\end{abstract}
\maketitle

\section{Introduction}\label{sec:intro}

Dirac-harmonic maps couple the harmonic map equation to a Dirac
equation for a spinor with values in the pullback tangent bundle
of the target. Introduced by Chen, Jost, Li, and Wang
\cite{CJLW06}, they are critical points of the conformally invariant
functional
\begin{equation}\label{eq:intro-functional}
 \mathcal L(\phi,\Psi)
 =\frac12\int_M\bigl(|d\phi|^2+
       \pair{\Psi}{\D^\phi\Psi}\bigr)\dvol_g.
\end{equation}
Here \((M,g)\) is a closed oriented Riemannian spin surface,
\((N,h)\) is a Riemannian manifold, and \(\Psi\) is a section of
\(\Sigma M\otimes_{\RR}\phi^*TN\). The Euler--Lagrange equations are
\begin{equation}\label{eq:DH}
 \D^\phi\Psi=0,\qquad \tau_g(\phi)=\cR(\phi,\Psi),
\end{equation}
where \(\cR\) contracts the target curvature with \(d\phi\) and
two spinor factors. We call a solution \emph{coupled} if
\(\tau_g(\phi)\not\equiv0\). The spinor then contributes a
nonzero term to the map equation, and the map is not harmonic.
A nonzero spinor alone does not ensure coupling.

This distinction is important in the existence theory. The
index-theoretic construction of Ammann--Ginoux \cite{AG13}
produces nonzero twisted harmonic spinors along harmonic maps,
and hence uncoupled solutions. Heat-flow methods also give
existence results with nonzero spinors; see Jost--Zhu
\cite{JZ21flow} and Jost--Sun--Zhu \cite{JSZ25}. These results
do not by themselves guarantee a nonharmonic map component.
For the \(\alpha\)-perturbed system, Jost--Zhu \cite{JZ21alpha}
obtain coupled solutions under a uniform energy hypothesis,
but the classical limit as \(\alpha\downarrow1\) need not remain
nontrivial.

Explicit constructions provide another approach. Jost--Mo--Zhu
\cite{JMZ09} constructed coupled examples on noncompact surfaces,
while Ammann--Ginoux \cite{AG19} studied constructions based on
immersions. Chen--Jost--Sun--Zhu \cite{CJSZ15} reduced the
Dirac-geodesic equation on a target surface to a
magnetic-geodesic equation involving its Gaussian curvature.
The latter work provides a useful geometric starting point, but
there is a further condition for a solution on a closed source:
the spinor must return with the sign prescribed by the source
spin structure. A closed Dirac-geodesic in the terminology of
\cite{CJSZ15} has a closed underlying curve; its spinor need not
close up. Repeated traversal need not restore the required spinor
periodicity.

We prove general existence results for coupled Dirac-harmonic
maps from closed spin tori. For a closed target surface with
a prescribed metric, we require only that its Gaussian
curvature is not identically zero. The construction works
on every flat torus with any prescribed spin structure;
uniformization and conformal covariance then yield the
result for arbitrary source metrics.

All manifolds and metrics below are smooth, and ``closed'' means
compact and without boundary; surfaces are assumed connected. Throughout, $\mathbb N=\{1,2,\ldots\}$.
We write
\[
 E_g(\phi)=\frac12\int_M|d\phi|_g^2\dvol_g
\]
for the map energy. Our first main result is the following.

\begin{theorem}[Nonflat surface targets]
\label{thm:nonflat-surface}
Let \((N,h)\) be a closed Riemannian surface. If its Gaussian
curvature \(K\) is not identically zero, then, for every closed
oriented Riemannian torus \((M,g)\) and every spin structure on
\(M\), there is a sequence of smooth Dirac-harmonic maps
\((\phi_j,\Psi_j):(M,g)\to(N,h)\) such that
\begin{equation}\label{eq:main-conclusions}
 \rank(d\phi_j)=1,\qquad
 \tau_g(\phi_j)\ne0\ \text{everywhere},\qquad
 E_g(\phi_j)\longrightarrow\infty.
\end{equation}
The maps \(\phi_j\) are null-homotopic. For a flat metric
conformal to \(g\), the corresponding twisted spinors are parallel.
No orientability assumption on \(N\) is
required.
\end{theorem}

No sign condition on \(K\), no nonzero total curvature, and no
symmetry of either metric is required. The maps have images in
small coordinate disks, so their nonzero tension is
supported by local target curvature even though they are
null-homotopic.

For sphere targets, this homotopy conclusion is sharp. Yang's
vanishing theorem \cite{Yang09}, also proved in \cite{CJSZ19},
implies that a Dirac-harmonic map from a torus to a Riemannian
two-sphere with nonzero degree is uncoupled, for every source
spin structure. Maps from a torus to \(S^2\) are classified by
their degree \cite{Hatcher02}, and every metric on \(S^2\) is
nonflat by Gauss--Bonnet. Thus, for any fixed source metric,
spin structure, and target metric on \(S^2\), a homotopy class
contains a coupled solution if and only if its degree is zero.

For strictly negative curvature, the same geometric mechanism
can be implemented in prescribed nontrivial homotopy classes.
The \emph{trivial spin structure} on a torus is the one admitting
nonzero parallel spinors for a flat conformal representative;
its spinor coefficients are periodic in both basic directions.

\begin{theorem}[Prescribed homotopy classes in negative curvature]
\label{thm:negative-homotopy}
Let \((N,h)\) be a closed oriented Riemannian surface with
\(K<0\), and let \((M,g)\) be a closed oriented Riemannian
torus with the trivial spin structure. Every homotopy class
\(\mathfrak h\in[M,N]\) contains a sequence of smooth coupled
Dirac-harmonic maps \((\phi_j,\Psi_j)\), with
\([\phi_j]=\mathfrak h\), satisfying
\eqref{eq:main-conclusions}. For a flat metric conformal to
\(g\), their corresponding twisted spinors are parallel.
\end{theorem}

Here the target metric is fixed and need not have constant
curvature. For the other source spin structures, our construction
requires a compatibility condition with the prescribed map class;
see \Cref{prop:compatible-homotopy}. The null-homotopic class is
available for every spin structure by \Cref{thm:nonflat-surface}.

The construction begins with a latitude of the round sphere.
A spinor phase cancels the rotation of the target frame, producing
a parallel twisted spinor. Along a general unit-speed surface
curve \(\gamma\), the map equation then reduces to
\begin{equation}\label{eq:intro-balance}
 \kappa=aK\circ\gamma,\qquad a\ne0\ \text{constant},
\end{equation}
the magnetic-geodesic relation in \cite{CJSZ15}.
The total geodesic curvature
\(\Theta_\gamma=\int_\gamma\kappa\,d\ell\) determines the
spinor phase after one traversal. Thus the curve must satisfy
both \eqref{eq:intro-balance} and the prescribed spinor closing
condition. Conformal covariance transfers the construction to
arbitrary source metrics.

For the metric \(dr^2+F(r)^2d\varphi^2\), the circles have
total turning \(2\pi F'(r)\). Where \(F''\ne0\), the ratio
\(\kappa/K\) is constant along each circle, and varying \(r\)
realizes the required phases, giving the following local result.

\begin{theorem}[A totally geodesic warped annulus]
\label{thm:warped-annulus}
Let
\[
 (P,h_P)=(I\times S^1,\ dr^2+F(r)^2d\varphi^2),
 \qquad F>0,\qquad F''\not\equiv0,
\]
be a two-dimensional warped annulus, where \(I\subset\RR\) is
an open interval and \(S^1=\RR/2\pi\ZZ\). Suppose a Riemannian
manifold \((N,h)\) admits a totally geodesic isometric immersion
of \((P,h_P)\). Then every closed Riemannian spin torus admits
a sequence of coupled Dirac-harmonic maps into \(N\) satisfying
\eqref{eq:main-conclusions}. For a flat conformal representative
of the source metric, the twisted spinors are parallel.
\end{theorem}

To prove \Cref{thm:nonflat-surface} without symmetry, we minimize
the energy of short loops at prescribed signed curvature area
\(\mathcal A_K(\gamma)=\int_\gamma\eta\), where
\(d\eta=K\,dA_h\) on a coordinate disk. The constraint serves
two purposes: its Euler--Lagrange equation gives
\eqref{eq:intro-balance}, and the identity
\[
 e^{\ii\Theta_\gamma}=e^{-\ii\mathcal A_K(\gamma)}
\]
controls the spinor phase. A sharp Poincar\'e estimate and comparison
with small circles force the minimizers to concentrate where
\(|K|\) is maximal. Their short lengths exclude closed geodesics,
ensuring nonzero geodesic curvature and hence coupling. Suitable
areas and traversal numbers give all source spin structures and
unbounded energy.

For \Cref{thm:negative-homotopy}, each nontrivial map class factors
through a closed target curve and an integral circle map on the
torus. We fix the circle map and vary the curve within its free
homotopy class, prescribing curvature area relative to the closed
geodesic in that class. This area is well defined because every
torus map into the target has degree zero. Strict negativity and
compactness keep \(|K|\) bounded away from zero. Varying the
relative area then gives the required phases and unbounded energy
without changing the map class.

We next record three consequences and extensions. First, adding
an independent harmonic circle component to the rank-one maps
gives coupled immersions. The product connection preserves the
spinor equation and the nonzero tension of the surface component.

\begin{corollary}[Coupled immersions into a product with a circle]
\label{cor:rank-two-product}
Let \((N,h)\) be a closed Riemannian surface with Gaussian
curvature \(K\not\equiv0\), and equip \(N\times S^1(1)\)
with the product metric. For every closed oriented Riemannian
torus \((M,g)\) and every spin structure on \(M\), there is a
sequence of smooth Dirac-harmonic maps
\((\Phi_j,\widehat\Psi_j):(M,g)\to N\times S^1(1)\) such that
\[
 \rank(d\Phi_j)=2,\qquad
 \tau_g(\Phi_j)\ne0\ \text{everywhere},\qquad
 E_g(\Phi_j)\longrightarrow\infty.
\]
For a flat metric conformal to \(g\), the corresponding twisted
spinors are parallel.
\end{corollary}

Second, the construction extends to certain higher-dimensional
targets with their prescribed metrics. The additional issue is
that covariant derivatives and curvature may have components
normal to the two target directions used in the spinor.
\Cref{thm:ambient-curve} states the conditions along a closed
curve that exclude these extra components. A direct consequence
of \Cref{thm:nonflat-surface} and composition with a totally
geodesic immersion is the following.

\begin{corollary}[Totally geodesic surface targets]
\label{cor:nonflat-TG-surface}
If \((N,h)\) admits a totally geodesic isometric immersion of a
closed nonflat Riemannian surface, then every closed Riemannian
spin torus admits null-homotopic coupled Dirac-harmonic maps into
\(N\) satisfying \eqref{eq:main-conclusions}.
\end{corollary}

This includes product targets and the warped products described
in \Cref{sec:higher}. A closed totally geodesic surface is not
always needed: the local annulus of
\Cref{thm:warped-annulus} already gives the next result.

\begin{corollary}[Nonflat space forms]\label{cor:spaceforms}
Let $(N,h)$ be a nonflat Riemannian manifold of constant
sectional curvature $c$. Every closed Riemannian spin torus admits
null-homotopic coupled Dirac-harmonic maps into \(N\) satisfying
\eqref{eq:main-conclusions}. Their images can be contained in one
normal coordinate ball. In particular, this applies to closed
spherical and closed hyperbolic manifolds.
\end{corollary}

Third, if the target metric may be chosen, such an annulus can
be placed in any prescribed open set of a manifold of dimension
at least two. No global restriction on the target topology is
then needed.

\begin{theorem}[A local choice of target metric]
\label{thm:metric-choice}
Let \(N\) be a smooth closed manifold with \(\dim N\ge2\),
\(h_{\mathrm{ref}}\) a Riemannian metric, and
\(\mathcal O\subset N\) a nonempty open set. There is a smooth
metric \(h\), equal to \(h_{\mathrm{ref}}\) outside
\(\mathcal O\), such that every closed Riemannian spin torus
admits null-homotopic coupled Dirac-harmonic maps into \((N,h)\)
satisfying \eqref{eq:main-conclusions}. All their images lie in one
fixed coordinate ball in \(\mathcal O\). The metric \(h\) is
chosen independently of the source metric and spin structure.
\end{theorem}

\medskip
\noindent\textbf{Related independent work.}
After completing the present work, we became aware of an
independently developed manuscript by
B\'ek\'esi~\cite{Bekesi26}, who constructs coupled
Dirac-harmonic maps from closed product manifolds, including
tori, using periodic Dirac-geodesics.
The present paper establishes existence for arbitrary
prescribed nonflat metrics on closed target surfaces and,
for strictly negatively curved targets, in prescribed
homotopy classes for tori with the trivial spin structure.
Our proofs use constrained variational constructions based
on local and relative curvature area.

\medskip
\noindent\textbf{Organization of the paper.}
\Cref{sec:prelim} fixes conventions and records conformal covariance
and composition with totally geodesic immersions.
\Cref{sec:examples,sec:annuli} develop the explicit examples and
the warped-annulus construction.
\Cref{sec:surfaces,sec:homotopy} prove
\Cref{thm:nonflat-surface,thm:negative-homotopy} using local and
relative curvature area, respectively.
\Cref{sec:higher,sec:metric-choice} treat higher-dimensional targets
with a fixed metric and the local choice of target metric.

\medskip
\noindent\textbf{Declaration of AI assistance.}
The authors directed the research, initially guiding ChatGPT
to construct examples from tori to the round two-sphere.
The authors soon recognized that the construction extends to
warped annuli and that the total turning is related to signed
curvature area. Building on this observation, the authors  used constrained
energy minimization and the Lagrange multiplier method to remove
the symmetry assumption and prove \Cref{thm:nonflat-surface}.
During the development of the proof of \Cref{thm:negative-homotopy},
ChatGPT helped identify relevant references on homotopy classes
of maps.
The authors verified the arguments, corrected errors, filled in
missing details, and organized the manuscript. The authors take
full responsibility for its mathematical content and proofs.

\section{Conventions and preliminary facts}\label{sec:prelim}
Let $(M,g)$ be an oriented Riemannian spin surface and let $(N,h)$ be a
Riemannian manifold. We use the full complex spinor bundle $\Sigma M$,
of complex rank two, with Hermitian inner product linear in the first
argument. Clifford multiplication and the spin connection
$\nabla^\Sigma$ satisfy
\begin{align}
 X\cdot Y\cdot+Y\cdot X\cdot&=-2g(X,Y)\id,
 \label{eq:clifford}\\
 \pair{X\cdot\xi}{\zeta}&=-\pair{\xi}{X\cdot\zeta},
 \label{eq:clifford-skew}\\
 \nabla^\Sigma_X(Y\cdot\xi)
 &=(\nabla^M_XY)\cdot\xi+Y\cdot\nabla^\Sigma_X\xi.
 \label{eq:clifford-compatible}
\end{align}
These are the standard Riemannian spin conventions; see \cite[Chapter~II]{LM89}.

For a smooth map $\phi:M\to N$, write $\nabla^\phi$ for the pullback
Levi--Civita connection on the real bundle $\phi^*TN$. The twisted
spinor bundle is
\[
 E_\phi=\Sigma M\otimes_{\RR}\phi^*TN.
\]
It is a complex bundle by scalar multiplication on the spinor factor:
$z(\xi\otimes A)=(z\xi)\otimes A$ for $z\in\CC$ and $A$ real.
Its Hermitian metric and connection are defined by
\begin{align}
 \pair{\xi\otimes A}{\zeta\otimes B}
 &=\pair{\xi}{\zeta}\,h(A,B),\label{eq:twisted-metric}\\
 \tnabla_X(\xi\otimes A)
 &=\nabla^\Sigma_X\xi\otimes A+
   \xi\otimes\nabla^\phi_XA.\label{eq:twisted-connection}
\end{align}
Clifford multiplication acts on the first factor.

The real tensor product used above has a natural complex-linear
identification
\begin{equation}\label{eq:complex-tensor}
 E_\phi\cong\Sigma M\otimes_{\CC}(\phi^*TN\otimes_{\RR}\CC).
\end{equation}
Under \eqref{eq:complex-tensor},
\(\xi\otimes(v\otimes z)\) corresponds to \((z\xi)\otimes v\).
Only the real vector bundle is complexified; no complex structure
on the target manifold is required. The compatible Hermitian
extension of \(h\) is
\[
 h_{\mathrm H}(zA,wB)=z\overline w\,h(A,B),
 \qquad A,B\in\phi^*TN,\quad z,w\in\CC.
\]
Thus \(\xi\otimes(U+\ii V)=\xi\otimes U+(\ii\xi)\otimes V\),
and for real orthonormal \(U,V\) its squared norm is \(2|\xi|^2\).

For a local orthonormal frame $(e_1,e_2)$ of $TM$, set
\begin{align}
 \D^\phi\Psi
 &=\sum_{\alpha=1}^2e_\alpha\cdot\tnabla_{e_\alpha}\Psi,
 \label{eq:dirac}\\
 \tau(\phi)
 &=\sum_{\alpha=1}^2
 \left\{\nabla^\phi_{e_\alpha}\bigl(d\phi(e_\alpha)\bigr)
       -d\phi(\nabla^M_{e_\alpha}e_\alpha)\right\}.
 \label{eq:tension}
\end{align}
Our target curvature convention is
\[
 R^N(X,Y)Z
 =\nabla^N_X\nabla^N_YZ-\nabla^N_Y\nabla^N_XZ-\nabla^N_{[X,Y]}Z.
\]
Consequently the unit sphere satisfies
\begin{equation}\label{eq:sphere-curvature}
 R^{S^2}(X,Y)Z=h(Y,Z)X-h(X,Z)Y.
\end{equation}
Define the real vector field $\cR(\phi,\Psi)$ along $\phi$ by
\begin{equation}\label{eq:curvature-term}
 h(\cR(\phi,\Psi),Y)
 =\frac12\sum_{\alpha=1}^2
 \Rea\pair{e_\alpha\cdot
 R^N(Y,d\phi(e_\alpha))\Psi}{\Psi},
\end{equation}
for every real target vector field $Y$ along $\phi$.
Here curvature acts on the target factor:
\[
 R^N(Y,A)(\xi\otimes B)=\xi\otimes R^N(Y,A)B.
\]

On a closed source, the twisted Dirac operator is formally
self-adjoint, so the spinor term in \eqref{eq:intro-functional}
has a real integral. No real-part symbol is therefore needed in
that functional.

Unless otherwise indicated, all maps and sections are smooth;
the variational competitors in \Cref{sec:surfaces,sec:homotopy} are initially
only of class \(H^1\).

In a real orthonormal target frame, write
\(\Psi=\sum_j\psi^j\otimes\mathsf e_j\).
Curvature symmetries and Clifford skew-adjointness give the
equivalent vector formula
\begin{equation}\label{eq:vector-curvature}
 \cR(\phi,\Psi)=\frac12\sum_{\alpha,j,k}
 \Rea\pair{\psi^j}{e_\alpha\cdot\psi^k}
 R(\mathsf e_j,\mathsf e_k)d\phi(e_\alpha).
\end{equation}
This agrees with \cite[Equation~(1.6)]{CJLW06}, with curvature
indices translated to our convention. We suppress the subscript
in \(E_g(\phi)\) when the source metric is understood.
On a solution, \(\mathcal L(\phi,\Psi)=E_g(\phi)\).

We record two natural operations used in the constructions.
\begin{lemma}[Conformal covariance]\label{lem:conformal}
Let \(\widehat g=e^{2u}g\) and let
\(\beta:\Sigma_gM\to\Sigma_{\widehat g}M\) be the canonical fiberwise
unitary identification for the same spin structure. Set
\(\widehat\Psi=e^{-u/2}\beta\Psi\), acting trivially on the target
factor. Then
\begin{align}
 \D^\phi_{\widehat g}\widehat\Psi
 &=e^{-3u/2}\beta(\D^\phi_g\Psi),\label{eq:D-conf}\\
 \tau_{\widehat g}(\phi)&=e^{-2u}\tau_g(\phi),\label{eq:tau-conf}\\
 \cR_{\widehat g}(\phi,\widehat\Psi)
 &=e^{-2u}\cR_g(\phi,\Psi).\label{eq:R-conf}
\end{align}
Coupling and map energy are preserved.
\end{lemma}
\begin{proof}
The Dirac and tension transformations are the standard
two-dimensional conformal covariance formulas; see
\cite[Lemma~3.1]{CJLW06} and \cite[Chapter~II]{LM89}.
For the curvature term, take
\(\widehat e_\alpha=e^{-u}e_\alpha\) and use
\(\widehat e_\alpha\cdot\beta\xi=\beta(e_\alpha\cdot\xi)\).
The two spinor factors in \eqref{eq:curvature-term} contribute
\(e^{-u}\), and \(d\phi(\widehat e_\alpha)\) contributes another
\(e^{-u}\), proving \eqref{eq:R-conf}.
The conformal factors in \(|d\phi|^2d\vol_g\) cancel as well.
\end{proof}

To regard a solution as taking values in a larger target, we use the following fact.
\begin{lemma}\label{lem:totally-geodesic}
Let \(\iota:(P,h_P)\to(N,h)\) be a totally geodesic isometric immersion.
If \((\phi,\Psi)\) solves \eqref{eq:DH} into \(P\), then
\((\iota\circ\phi,(\id\otimes d\iota)\Psi)\) solves it into \(N\).
The tension is \(d\iota(\tau(\phi))\), so coupling is preserved.
\end{lemma}
\begin{proof}
Identify \(TP\) with its image. Since its second fundamental form
vanishes, \(\nabla^N_XZ=\nabla^P_XZ\) for tangent fields. Thus the
twisted connections, tensions, and tangential curvature terms agree.
For normal \(Y\) and tangent \(X,Z,W\), curvature symmetry gives
\[
 h(R^N(Y,X)Z,W)=h(R^N(Z,W)Y,X)
              =-h(Y,R^N(Z,W)X)=0.
\]
The last equality follows because tangent covariant differentiation
preserves \(TP\). Substituting this in \eqref{eq:curvature-term} proves
that the normal curvature-force component vanishes.
\end{proof}

\section{Explicit examples on closed tori}\label{sec:examples}
\subsection{Spinors on the flat torus}\label{subsec:flat-spinors}

For $a_0>0$, define
\begin{equation}\label{eq:torus}
 M_{a_0}=(\RR/2\pi\ZZ)^2,\qquad
 g_{a_0}=a_0^2dx^2+dy^2,\qquad
 e_1=a_0^{-1}\partial_x,\quad e_2=\partial_y.
\end{equation}
The orientation is that of $dx\wedge dy$. The displayed tangent frame
is global, orthonormal, and parallel.
We write \([x,y]\) for the point of the torus represented by
the angular coordinates \(x,y\). Differentiation uses local real
representatives of these coordinates; \(dx,dy\) and the coordinate
vector fields are globally defined. In the flat-torus constructions,
spinor formulas are local expressions in parallel spinor frames.
Their transition signs will be checked to obtain global sections.

On a simply connected flat coordinate neighborhood, choose a
parallel unitary spinor frame induced by a local spin lift of the
displayed tangent frame. Clifford multiplication by the
\emph{fixed} tangent vectors \(e_1,e_2\) is represented by
\begin{equation}\label{eq:matrices}
 C_1=\begin{pmatrix}\ii&0\\0&-\ii\end{pmatrix},
 \qquad
 C_2=\begin{pmatrix}0&1\\-1&0\end{pmatrix}.
\end{equation}
The choice concerns a frame of $\Sigma M_{a_0}$, not a change of the
tangent frame $(e_1,e_2)$.

Indeed, choose a unit \(+\ii\)-eigenspinor \(\varepsilon_1\)
of \(e_1\cdot\) and set \(\varepsilon_2=-e_2\cdot\varepsilon_1\).
The Clifford relations make this a unitary basis with matrices
\eqref{eq:matrices}. Parallel transport within the coordinate
neighborhood preserves these matrices by
\eqref{eq:clifford-compatible}.

If $\xi$ denotes the column of spinor coefficients in this frame, then
\begin{equation}\label{eq:spin-coordinates}
 e_\alpha\cdot\xi=C_\alpha\xi,\qquad
 \nabla^\Sigma_{e_1}\xi=a_0^{-1}\partial_x\xi,\qquad
 \nabla^\Sigma_{e_2}\xi=\partial_y\xi.
\end{equation}
Relative to the parallel tangent frame, the four spin structures are
indexed by
\[
 (\sigma,\delta)\in\left\{0,\tfrac12\right\}^2.
\]
The transition signs of the local spinor frames along the two
fundamental directions require the coefficient columns to satisfy
\begin{align}
 \xi(x+2\pi,y)&=e^{2\pi\ii\sigma}\xi(x,y),\label{eq:spin-x}\\
 \xi(x,y+2\pi)&=e^{2\pi\ii\delta}\xi(x,y).\label{eq:spin-y}
\end{align}
These are the transition rules for local spinor coefficients:
\(0\) gives a periodic sign and \(1/2\) an antiperiodic sign.
They also apply to the spinor coefficients of a twisted section
in a global target frame along the map. Since the signs are
constant, the local connection formulas agree on overlaps.

Set
\begin{equation}\label{eq:chi}
 \chi=\binom10,\qquad |\chi|=1,\qquad C_1\chi=\ii\chi.
\end{equation}
For $q\in\RR$, the nonzero spinor coefficient $e^{\ii qx}\chi$
satisfies the prescribed identifications exactly when
\begin{equation}\label{eq:frequency}
 q\in\ZZ+\sigma,\qquad \delta=0.
\end{equation}
The two translations multiply the coefficient by
\(e^{2\pi\ii q}\) and \(1\). The local parallel spinor
\(\chi\) itself is global only for the trivial spin structure.

\subsection{The latitude construction}

Let $m\in\mathbb{N}$ and $0<\theta<\pi/2$. Write
\begin{equation}\label{eq:parameters}
 s=\sin\theta,\qquad c=\cos\theta,\qquad
 a_0=ms,\qquad q=mc,\qquad \kappa=\frac cs=\frac{q}{a_0}>0.
\end{equation}
The target is $N=S^2(1)\times S^1(1)$ with its product metric, where
$S^2(1)\subset\RR^3$ and $S^1(1)\subset\CC$ have their standard
induced metrics. Define
\begin{align}
 \gamma(x)&=(s\cos(mx),s\sin(mx),c),\label{eq:latitude}\\
 \phi([x,y])&=(\gamma(x),e^{\ii y}):M_{a_0}\longrightarrow N.
 \label{eq:map}
\end{align}
The coordinates in \eqref{eq:latitude} are Euclidean coordinates in
$\RR^3$. The parameter $\theta$ is the polar angle measured from the
positive third axis. Thus the image of $\gamma$ is the latitude circle
at height $c$, with Euclidean radius $s$. As $x$ runs from $0$ to
$2\pi$, the curve traverses that circle $m$ times. In particular,
the integer condition on $m$ makes $\phi$ globally defined.

Along the sphere component, set
\begin{align}
 U(x)&=(c\cos(mx),c\sin(mx),-s),\label{eq:U}\\
 V(x)&=(-\sin(mx),\cos(mx),0).\label{eq:V}
\end{align}
Direct Euclidean inner products give
\[
 \pair{U}{\gamma}=\pair{V}{\gamma}=\pair{U}{V}=0,
 \qquad |U|=|V|=|\gamma|=1.
\]
Thus $U,V$ are a real orthonormal frame of the sphere tangent plane
along $\gamma$. The vector $U$ points along increasing polar angle,
and $V$ points along increasing azimuth. We also write $U,V$ for their
lifts $(U,0),(V,0)$ to the product target, and set
\[
 W(y)=(0,\ii e^{\ii y}).
\]
Then $(U,V,W)$ is a real orthonormal frame of $\phi^*TN$, periodic in
both variables.

\begin{lemma}\label{lem:geometry}
The map $\phi$ is an isometric immersion and satisfies
\begin{align}
 d\phi(e_1)&=V,& d\phi(e_2)&=W,\label{eq:dphi}\\
 \nabla^\phi_{e_1}U&=\kappa V,&
 \nabla^\phi_{e_1}V&=-\kappa U,&
 \nabla^\phi_{e_1}W&=0,\label{eq:connection-x}\\
 \nabla^\phi_{e_2}U&=0,&
 \nabla^\phi_{e_2}V&=0,&
 \nabla^\phi_{e_2}W&=0.\label{eq:connection-y}
\end{align}
Its tension field is
\begin{equation}\label{eq:tension-value}
 \tau(\phi)=-\kappa U.
\end{equation}
For $m=1$ it is an embedding; for $m>1$ it covers the corresponding
embedded latitude torus $m$ times in the first direction.
\end{lemma}

\begin{proof}
Differentiation in \(\RR^3\) gives
\[
 \partial_x\gamma=a_0V,\qquad \partial_xU=qV,
\]
and, using \(cU+s\gamma=(\cos(mx),\sin(mx),0)\),
\begin{equation}\label{eq:ambient-V}
 \partial_xV=-qU-a_0\gamma.
\end{equation}
For a tangent field \(Z\) along the sphere, tangential projection
of its Euclidean derivative gives
\begin{equation}\label{eq:projection}
 \nabla^\gamma_{\partial_x}Z
 =\partial_xZ+\pair{Z}{\partial_x\gamma}\gamma.
\end{equation}
Consequently \(\nabla^\gamma_{\partial_x}U=qV\) and
\(\nabla^\gamma_{\partial_x}V=-qU\). Since
\(e_1=a_0^{-1}\partial_x\) and \(q/a_0=\kappa\), this proves
the first two identities in \eqref{eq:connection-x}.
In arc length, \(d\gamma/d\ell=V\) and
\(\nabla^{S^2}_VV=-\kappa U\), so \(\kappa=\cot\theta\)
is the signed geodesic curvature of the latitude.

The sphere fields are independent of $y$. On the circle factor,
$\partial_y(\ii e^{\ii y})=-e^{\ii y}$ is normal to the circle, so
the intrinsic covariant derivative of its unit tangent is zero.
The product connection has no mixed terms. This proves every
remaining identity in \eqref{eq:connection-x}--\eqref{eq:connection-y}.

The derivatives of $\phi$ are $d\phi(\partial_x)=a_0V$ and
$d\phi(\partial_y)=W$. Their lengths and orthogonality give
$\phi^*h=a_0^2dx^2+dy^2=g_{a_0}$, so $\phi$ is an isometric immersion.
Since the source frame is parallel, \eqref{eq:tension} becomes
\[
 \tau(\phi)=\nabla^\phi_{e_1}V+\nabla^\phi_{e_2}W=-\kappa U.
\]
For \(m=1\) this is the product of two embedded circles;
the general map is its pullback by \([x,y]\mapsto[mx,y]\).
\end{proof}

We now choose the spinor phase to cancel these target connection terms.

Assume
\begin{equation}\label{eq:resonance}
 q=m\cos\theta\in\ZZ+\sigma,\qquad
 \sigma\in\left\{0,\tfrac12\right\},
\end{equation}
and equip $M_{a_0}$ with spin structure $(\sigma,0)$. In the local
parallel spinor frames of \Cref{subsec:flat-spinors}, define
\begin{equation}\label{eq:spinor}
 \xi=\sqrt{\kappa}\,e^{\ii qx}\chi,\qquad
 \Psi=\xi\otimes U+(\ii\xi)\otimes V.
\end{equation}
Both target fields $U,V$ are real and periodic. Moreover,
\[
 \xi(x+2\pi,y)=e^{2\pi\ii q}\xi(x,y)
             =e^{2\pi\ii\sigma}\xi(x,y),
 \qquad \xi(x,y+2\pi)=\xi(x,y).
\]
Thus the local coefficients obey the spin identifications and
define a global smooth section \(\Psi\) of \(E_\phi\).

\begin{lemma}\label{lem:parallel}
The twisted spinor \eqref{eq:spinor} satisfies
\begin{equation}\label{eq:parallel}
 \tnabla\Psi=0,\qquad \D^\phi\Psi=0,\qquad |\Psi|^2=2\kappa.
\end{equation}
\end{lemma}

\begin{proof}
Using $q/a_0=\kappa$ and \eqref{eq:spin-coordinates}, we have
\[
 \nabla^\Sigma_{e_1}\xi=\frac{\ii q}{a_0}\xi=\ii\kappa\xi,
 \qquad
 \nabla^\Sigma_{e_2}\xi=0.
\]
The spin connection is complex linear, so
$\nabla^\Sigma_{e_1}(\ii\xi)=-\kappa\xi$.
The tensor product rule and \eqref{eq:connection-x} now give
\begin{align*}
 \tnabla_{e_1}\Psi
 ={}&(\ii\kappa\xi)\otimes U
       +\kappa\xi\otimes V\\
    &-\kappa\xi\otimes V
       -\kappa(\ii\xi)\otimes U
 =0.
\end{align*}
Every term in $\tnabla_{e_2}\Psi$ vanishes by
\eqref{eq:connection-y} and $\nabla^\Sigma_{e_2}\xi=0$.
This proves parallelness, and hence $\D^\phi\Psi=0$.
Finally, the mixed term in the norm vanishes because $h(U,V)=0$.
Therefore \eqref{eq:twisted-metric} gives
\[
 |\Psi|^2=|\xi|^2+|\ii\xi|^2=2\kappa.
\]
\end{proof}

It remains to compute the curvature force in every target direction.

Since the target metric is a product and the circle factor has zero
curvature,
\begin{equation}\label{eq:product-curvature}
 R^N(Y,W)=0\quad\hbox{for all }Y,\qquad R^N(W,V)=0.
\end{equation}
On the sphere factor, \eqref{eq:sphere-curvature} gives
\begin{equation}\label{eq:sphere-actions}
 R^N(U,V)U=-V,\qquad R^N(U,V)V=U.
\end{equation}

\begin{lemma}\label{lem:force}
For the pair \eqref{eq:map}, \eqref{eq:spinor},
\begin{equation}\label{eq:force}
 \cR(\phi,\Psi)=-\kappa U.
\end{equation}
\end{lemma}

\begin{proof}
Because $C_1\chi=\ii\chi$, we have $e_1\cdot\xi=\ii\xi$.
By complex linearity of Clifford multiplication,
$e_1\cdot(\ii\xi)=-\xi$. Consequently
\[
 e_1\cdot\Psi=(\ii\xi)\otimes U-\xi\otimes V=\ii\Psi.
\]
Independently, \eqref{eq:sphere-actions} gives
\[
 R^N(U,V)\Psi=-\xi\otimes V+(\ii\xi)\otimes U=\ii\Psi.
\]
The two actions commute, since they act on different factors.
It follows that
\[
 e_1\cdot R^N(U,V)\Psi
 =e_1\cdot(\ii\Psi)=\ii(e_1\cdot\Psi)=-\Psi.
\]

We now evaluate \eqref{eq:curvature-term} against each member of the
real orthonormal target frame $(U,V,W)$.
The $\alpha=2$ contribution is zero for every $Y$, since
$d\phi(e_2)=W$ and \eqref{eq:product-curvature} applies.
For the remaining contribution with $Y=U$, we obtain
\[
 h(\cR(\phi,\Psi),U)
 =\frac12\Rea\pair{e_1\cdot R^N(U,V)\Psi}{\Psi}
 =-\frac12|\Psi|^2=-\kappa.
\]
For $Y=V$, the contribution is zero because $R^N(V,V)=0$.
For $Y=W$, it is zero because $R^N(W,V)=0$.
These are all three components of $\cR(\phi,\Psi)$, and prove
\eqref{eq:force}.
\end{proof}

The preceding calculations give the following closed family.

\begin{theorem}\label{thm:family}
Let $m\in\mathbb{N}$ and $\sigma\in\{0,\tfrac12\}$. Choose
\[
 q\in(\ZZ+\sigma)\cap(0,m),
\]
and define
\begin{equation}\label{eq:chosen-parameters}
 \theta=\arccos(q/m),\qquad
 a_0=\sqrt{m^2-q^2},\qquad
 \kappa=\frac{q}{\sqrt{m^2-q^2}}.
\end{equation}
Equip $(M_{a_0},g_{a_0})$ with spin structure $(\sigma,0)$. Then
\begin{align}
 \phi([x,y])
 &=\left(\left(\frac{a_0}{m}\cos(mx),\frac{a_0}{m}\sin(mx),\frac qm\right),
                   e^{\ii y}\right),\label{eq:family-map}\\
 \Psi
 &=\sqrt{\kappa}\,e^{\ii qx}\chi\otimes U
    +\ii\sqrt{\kappa}\,e^{\ii qx}\chi\otimes V
 \label{eq:family-spinor}
\end{align}
is a coupled Dirac-harmonic map into $S^2(1)\times S^1(1)$.
Here $U,V$ are given by \eqref{eq:U}--\eqref{eq:V} with the stated
$\theta$. The map is an isometric immersion, and
\begin{equation}\label{eq:family-conclusions}
 \tnabla\Psi=0,\qquad
 \tau(\phi)=\cR(\phi,\Psi)=-\kappa U\ne0,\qquad
 |\Psi|^2=2\kappa.
\end{equation}
Both source and target are closed.
\end{theorem}

\begin{proof}
The parameter choices give \(q=m\cos\theta\) and
\(\kappa=q/a_0>0\). Apply
\Cref{lem:geometry,lem:parallel,lem:force}.
\end{proof}

The spinor normalization is determined by the map equation. If
$\sqrt{\kappa}$ in \eqref{eq:spinor} is replaced by any nonzero
complex constant $A$, the same derivative calculation still gives
$\tnabla\Psi=0$. Its squared norm is then $2|A|^2$, and the curvature
calculation gives $\cR(\phi,\Psi)=-|A|^2U$. Hence this fixed map
solves the system with that spinor precisely when $|A|^2=\kappa$.
A constant unit complex phase is free.

\subsection{Concrete examples}

\begin{example}[An embedded coupled torus]\label{ex:embedded}
Take
\[
 m=1,\qquad \sigma=\frac12,\qquad q=\frac12.
\]
Then
\[
 \theta=\frac{\pi}{3},\qquad
 a_0=\frac{\sqrt3}{2},\qquad
 \kappa=\frac1{\sqrt3}.
\]
The source and map are
\begin{align}
 M&=(\RR/2\pi\ZZ)^2,\qquad
 g=\frac34\,dx^2+dy^2,\label{eq:embedded-metric}\\
 \phi([x,y])
 &=\left(\left(\frac{\sqrt3}{2}\cos x,
              \frac{\sqrt3}{2}\sin x,\frac12\right),e^{\ii y}\right).
 \label{eq:embedded-map}
\end{align}
The spin structure is antiperiodic in $x$ and periodic in $y$.
Set
\[
 U(x)=\left(\frac12\cos x,\frac12\sin x,-\frac{\sqrt3}{2}\right),
 \qquad V(x)=(-\sin x,\cos x,0),
\]
and define the twisted spinor by
\begin{equation}\label{eq:embedded-spinor}
 \xi=3^{-1/4}e^{\ii x/2}\binom10,\qquad
 \Psi=\xi\otimes U+(\ii\xi)\otimes V.
\end{equation}
The equality $e^{\ii(x+2\pi)/2}=-e^{\ii x/2}$ verifies the
antiperiodic identification explicitly. Therefore this formula is
global in the specified spin structure.
By \Cref{thm:family},
\[
 \tnabla\Psi=0,\qquad
 \tau(\phi)=\cR(\phi,\Psi)=-\frac1{\sqrt3}U,\qquad
 |\Psi|^2=\frac2{\sqrt3}.
\]
The map is an embedding by \Cref{lem:geometry}.
\end{example}

The same underlying latitude and its product with a circle
also arise from the constructions in
\cite[Sections~2--3]{Bekesi26}.

\begin{example}[A double cover with periodic spin structure]
\label{ex:periodic}
Pull back \Cref{ex:embedded} by
\([x,y]\mapsto[2x,y]\). This gives \(m=2\), \(q=1\),
the metric \(3\,dx^2+dy^2\), and coefficient
\(\xi=3^{-1/4}e^{\ii x}\chi\). The first spin sign is now
positive, so the spin structure is periodic in both directions.
The map is a two-fold isometric covering of the embedded torus,
and \(\tau(\phi)=\cR(\phi,\Psi)=-U/\sqrt3\).
\end{example}

The circle factor contributes neither tension nor curvature
force, and the spinor has no component in that direction.
Removing it gives the following rank-one solutions.

\begin{proposition}[Coupled examples with target $S^2$]
\label{prop:sphere-target}
For any parameters in \Cref{thm:family}, retain the same source metric
and spin structure, and define
\[
 \phi_0([x,y])=\gamma(x):M_{a_0}\longrightarrow S^2(1),\qquad
 \Psi_0=\xi\otimes U+(\ii\xi)\otimes V,
 \qquad \xi=\sqrt{\kappa}\,e^{\ii qx}\chi.
\]
Then
\[
 \tnabla\Psi_0=0,\qquad
 \tau(\phi_0)=\cR(\phi_0,\Psi_0)=-\kappa U\ne0,\qquad
 \rank(d\phi_0)=1.
\]
In particular, these are coupled Dirac-harmonic maps from a closed
torus to the unit sphere.
\end{proposition}

\begin{proof}
The sphere connection and the spinor are unchanged, so
\Cref{lem:parallel} applies. Since
\(d\phi_0(e_1)=V\) and \(d\phi_0(e_2)=0\), the tension and
the nonzero curvature contribution are those of
\Cref{lem:geometry,lem:force}. The derivative has rank one.
\end{proof}

\section{From the latitude example to warped annuli}
\label{sec:annuli}

\subsection{The mechanism of the example}
The rank-one example in \Cref{prop:sphere-target} already contains
the construction we need. Its essential observation is that a suitable
spinor phase cancels the rotation of the target frame. The amplitude
then balances the tension against the curvature force, while the spin
structure determines which phases are globally admissible.

To see the cancellation directly, write \(\varphi=mx\) for the longitude
of the latitude; \(\theta\) remains its fixed colatitude. The connection
formulas in \Cref{lem:geometry} become
\begin{equation}\label{eq:latitude-turn}
 \nabla_{\partial_\varphi}U=\cos\theta\,V,\qquad
 \nabla_{\partial_\varphi}V=-\cos\theta\,U.
\end{equation}
Use arc length \(\ell=\sin\theta\,\varphi\), and let
\(P_1,P_2\) be a parallel oriented orthonormal frame along
the curve with \(\ell\in\RR\), initially equal to \(U,V\).
Equation \eqref{eq:latitude-turn} gives
\[
 \begin{gathered}
 U(\ell)+\ii V(\ell)
 =e^{-\ii\alpha(\ell)}\bigl(P_1(\ell)+\ii P_2(\ell)\bigr),\\
 \alpha(\ell)=\ell\cot\theta,\qquad
 \alpha(2\pi\sin\theta)-\alpha(0)=2\pi\cos\theta.
 \end{gathered}
\]
Pull these fields back to the source by the local arc-length
function \(\ell=m\sin\theta\,x\). For a local parallel spinor
\(\chi\) on the flat torus and a constant \(A\), this gives
\begin{equation}\label{eq:phase-cancellation}
 \begin{aligned}
 \Psi&=Ae^{\ii\alpha(\ell)}\chi
             \otimes\bigl(U(\ell)+\ii V(\ell)\bigr)\\
     &=A\chi\otimes\bigl(P_1(\ell)+\ii P_2(\ell)\bigr).
 \end{aligned}
\end{equation}
The right-hand side is parallel, so \(\tnabla\Psi=0\) and
\(\D^\phi\Psi=0\) locally. Since
\(\alpha(m\sin\theta\,x)=m\cos\theta\,x\), this recovers the
phase \(e^{\ii qx}\) with \(q=m\cos\theta\). We use
\(\alpha\) for the accumulated phase and \(q\) for its constant
angular frequency. Global well-definedness requires
\eqref{eq:spin-x}--\eqref{eq:spin-y}.

The same observation gives a concise criterion for a general curve.
Let \(\gamma(\ell)\) be a smooth curve parametrized by arc length in
an oriented surface. Retaining the notation of the example, define
\[
 V(\ell):=\frac{d\gamma}{d\ell}(\ell)\in T_{\gamma(\ell)}N,
 \qquad |V(\ell)|_h=1,\qquad U(\ell):=-JV(\ell),
\]
where \(J\) is rotation by \(+\pi/2\). For a field \(Z\)
along the curve,
\(\nabla_VZ\) abbreviates \(\nabla^\gamma_{\partial_\ell}Z\).
The choice of \(U\) agrees with the outward radial vector in the
latitude example. The signed geodesic curvature \(\kappa\) satisfies
\[
 \nabla_VU=\kappa V,\qquad \nabla_VV=-\kappa U.
\]
In an orientation-preserving regular parameter \(s\), these relations give
\[
 V=\frac{\dot\gamma}{|\dot\gamma|_h},\qquad
 \nabla_sV=|\dot\gamma|_h\kappa JV,\qquad
 d\ell=|\dot\gamma|_h\,ds.
\]
Under an orientation-preserving change \(s=s(u)\), both the
speed and the covariant derivative of \(V\) are multiplied by
\(s'(u)>0\). Thus the signed curvature transforms as
\(\widetilde\kappa(u)=\kappa(s(u))\), and change of variables
shows that the total turning \(\int\kappa\,d\ell\) over one
traversal is unchanged.
Suppose locally that \(\phi=\gamma\circ\ell\), where \(\ell\)
denotes the arc-length parameter pulled back to the source.
For a closed curve, a local real representative of this parameter
is understood. Let \(E_1,E_2\) be a parallel orthonormal source
frame with \(d\ell(E_1)=v>0\) and \(d\ell(E_2)=0\), where
\(v\) is constant. Then \(d\phi(E_1)=vV\) and
\(d\phi(E_2)=0\). Fields and functions along the curve in the
following local formulas are evaluated at the source value of
\(\ell\).
Choose a local parallel unit spinor \(\chi_\epsilon\) with
\(E_1\cdot\chi_\epsilon=\ii\epsilon\chi_\epsilon\),
\(\epsilon=\pm1\), and put
\[
 \Psi=A(\ell)e^{\ii\alpha(\ell)}\chi_\epsilon
          \otimes\bigl(U(\ell)+\ii V(\ell)\bigr),
 \qquad A>0.
\]
Primes on \(A\) and \(\alpha\) denote differentiation with
respect to the curve parameter \(\ell\). All coefficients are independent of the \(E_2\) direction. Direct
differentiation gives
\[
 \begin{aligned}
 \tnabla_{E_1}\Psi
 &=v\bigl(A'(\ell)+\ii A(\ell)
                      (\alpha'(\ell)-\kappa(\ell))\bigr)\\
 &\qquad\cdot e^{\ii\alpha(\ell)}\chi_\epsilon
                    \otimes\bigl(U(\ell)+\ii V(\ell)\bigr).
 \end{aligned}
\]
Since \(E_1\cdot\) is invertible, the Dirac equation is therefore
equivalent to \(A'=0\) and \(\alpha'=\kappa\).
We normalize the phase by \(\alpha(0)=0\), so
\(\alpha(\ell)=\int_0^\ell\kappa(s)\,ds\). This notation for
the phase is retained in the subsequent constructions.
If \(K\) is the target Gaussian curvature along the curve, the
calculation of \Cref{lem:force} gives
\begin{equation}\label{eq:curve-balance}
 \tau(\phi)=-v^2\kappa U,\qquad
 \cR(\phi,\Psi)=-\epsilon vKA^2U.
\end{equation}
Indeed, \(R(U,V)\Psi=\ii K\Psi\) and \(|\Psi|^2=2A^2\).
Thus the remaining pointwise condition is
\begin{equation}\label{eq:curve-ratio}
 \kappa=\frac{\epsilon A^2}{v}K.
\end{equation}
Where \(K\ne0\), this asks that \(\kappa/K\) be constant along
the curve. Finally, after \(m\) traversals of a curve of length \(L\),
the multiplier
\[
 \exp\left(\ii m\int_0^L\kappa\,d\ell\right)
\]
must equal the prescribed source spin sign. Connection cancellation,
curvature balance, and this global phase condition are the three
ingredients retained below.

\subsection{The warped-annulus theorem}
The latitude construction uses only the target geometry along the
latitude and in a neighborhood of it; the rest of the sphere plays
no role. In polar coordinates that neighborhood has metric
\[
 dr^2+\sin^2r\,d\varphi^2,
\]
and the original latitude is \(r=\theta\). This suggests replacing
\(\sin r\) by a positive function while retaining the circle
coordinate \(\varphi\). More precisely, consider
\begin{equation}\label{eq:warped-annulus}
 P=I\times_F S^1,\qquad h_P=dr^2+F(r)^2d\varphi^2,\qquad
 \varphi\in\RR/2\pi\ZZ.
\end{equation}
Here \(I\subset\RR\) is the open interval of radial coordinates,
and \(F:I\to(0,\infty)\) is smooth.
Shifts of \(\varphi\) are isometries; their orbits are the circles
\(r=\text{constant}\), of length \(2\pi F(r)\).

Define the orthonormal frame on the annulus by
\(U=\partial_r\) and \(V=F^{-1}\partial_\varphi\). For a fixed
radius \(r\), its orbit has the unit-speed parametrization
\[
 \gamma_r:\RR/L_r\ZZ\longrightarrow P,\qquad
 \gamma_r(\ell)=\bigl(r,[\ell/F(r)]\bigr),\qquad L_r=2\pi F(r),
\]
where brackets reduce the angular coordinate modulo \(2\pi\).
The chain rule gives
\[
 \frac{d\gamma_r}{d\ell}(\ell)
 =F(r)^{-1}(\partial_\varphi)_{\gamma_r(\ell)}
 =V_{\gamma_r(\ell)}.
\]
Thus restricting this annulus field to an orbit gives exactly the
unit tangent used above. Orient the annulus by
\(dr\wedge d\varphi\), so that \(JU=V\) and \(U=-JV\).
Primes on \(F\) denote derivatives with respect to \(r\).
The Christoffel formula applied to \(h_{rr}=1\) and
\(h_{\varphi\varphi}=F^2\) gives
\[
 \Gamma^r_{\varphi\varphi}
   =-\frac12\partial_r(F^2)=-FF',\qquad
 \Gamma^\varphi_{r\varphi}
   =\frac{1}{2F^2}\partial_r(F^2)=\frac{F'}F.
\]
In the orthonormal frame, these identities become
\begin{equation}\label{eq:warp-connection}
 \nabla_UU=\nabla_UV=0,\qquad
 \nabla_VU=\frac{F'}F V,\qquad
 \nabla_VV=-\frac{F'}F U.
\end{equation}
In particular, the unit-speed orbit has curvature vector
\(-(F'/F)U\), and its signed geodesic curvature is
\(\kappa=h(\nabla_VV,JV)=F'/F\). Geometrically,
\[
 \frac{F'(r)}{F(r)}=\frac{d}{dr}\log L_r,
 \qquad L_r=2\pi F(r),
\]
so this curvature is the relative radial growth rate of the
orbit length. For \(F(r)=r\), it reduces to the curvature
\(1/r\) of a Euclidean circle. Since \([U,V]=-(F'/F)V\),
\[
 R(U,V)V
 =-\left(\left(\frac{F'}F\right)'
              +\left(\frac{F'}F\right)^2\right)U
 =-\frac{F''}F U.
\]
Consequently, along the orbit at radius \(r\),
\begin{equation}\label{eq:warp-data}
 \kappa=\frac{F'}F,\qquad K=-\frac{F''}F,\qquad
 \int_{\gamma_r}\kappa\,d\ell=2\pi F'(r).
\end{equation}
Both \(\kappa\) and \(K\) are constant on that orbit, so
\eqref{eq:curve-ratio} can be achieved by choosing the amplitude
and the sign \(\epsilon\). If \(F''\ne0\), the total turning
\(2\pi F'(r)\) varies with \(r\), allowing the spin condition to
be met by a nearby orbit. Thus the annulus supplies both the local
geometry used in the example and the family of closed curves needed
for global compatibility. A center is unnecessary.
By \Cref{lem:totally-geodesic}, a solution in this annulus also gives
a solution in any target containing it totally geodesically, which
accounts for the immersion hypothesis in \Cref{thm:warped-annulus}.

\begin{proof}[Proof of \Cref{thm:warped-annulus}]
\noindent\textbf{Step 1: Choose the source parameter and frame.}
By uniformization and the lattice description of conformal tori
\cite[Theorem~4.4.1 and \S2.7]{Jost06}, after an oriented
diffeomorphism we may write
\begin{equation}\label{eq:general-flat-torus}
 M=(\RR/2\pi\ZZ)^2,\qquad g=e^{2u}g_0,\qquad
 g_0=G_{11}dx^2+2G_{12}dx\,dy+G_{22}dy^2,
\end{equation}
Here \(G=(G_{ij})\) is a constant symmetric positive definite
matrix. Thus \(g_0\) is flat and \(\partial_x,\partial_y\) are
parallel. The coefficients are allowed to be arbitrary, so the
given conformal class is retained.
By \Cref{lem:conformal,lem:totally-geodesic}, it suffices to construct
solutions for \(g_0\) with target \(P\).

Recall from \Cref{subsec:flat-spinors} that the phase \(e^{\ii qx}\)
realizes the spin signs \((\sigma,\delta)\) only when
\(q\in\ZZ+\sigma\) and \(\delta=0\); see \eqref{eq:frequency}.
To accommodate all four spin structures, replace \(x\) by the
circle-valued map, with \(S^1=\RR/2\pi\ZZ\),
\begin{equation}\label{eq:circle-parameter}
 t:M\longrightarrow S^1,\qquad
 t([x,y])=[p_1x+p_2y]_{2\pi},\qquad
 p=(p_1,p_2)\in\ZZ^2\setminus\{0\}.
\end{equation}
Brackets denote reduction modulo the indicated period.
The local representative of \(t\) is \(p_1x+p_2y\), and
\(p_i\) is its degree along the \(i\)-th coordinate loop.
For the phase \(e^{\ii q(p_1x+p_2y)}\), the transition rules
\eqref{eq:spin-x}--\eqref{eq:spin-y} require
\begin{equation}\label{eq:spin-congruences}
 qp_1\equiv\sigma\pmod{\ZZ},\qquad
 qp_2\equiv\delta\pmod{\ZZ}.
\end{equation}
To satisfy these conditions for every \(q\in\ZZ+\sigma_*\),
choose
\begin{equation}\label{eq:choose-direction}
 (p,\sigma_*)=
 \begin{cases}
 ((1,0),0),&(\sigma,\delta)=(0,0),\\
 ((2\sigma,2\delta),\tfrac12),&(\sigma,\delta)\ne(0,0).
 \end{cases}
\end{equation}
Indeed, this gives \(\sigma_*p_1=\sigma\) and
\(\sigma_*p_2=\delta\).

The local representatives of \(t\) define the global one-form
\(dt=p_1dx+p_2dy\). Its norm \(\rho:=|dt|_{g_0}>0\) is constant,
since the coefficients of both \(dt\) and \(g_0\) are constant.
Let \((dt)^{\sharp_{g_0}}\) denote its metric-dual vector field,
and let \(J_0\) be rotation by \(+\pi/2\) on the oriented
source surface \((M,g_0)\). Define
\begin{equation}\label{eq:adapted-source-frame}
 E_1=\rho^{-1}(dt)^{\sharp_{g_0}},\qquad E_2=J_0E_1.
\end{equation}
Since \(dt\) and \(J_0\) are parallel, these fields form a
positively oriented parallel orthonormal frame. Moreover,
\[
 dt(E_1)=\rho^{-1}|dt|_{g_0}^2=\rho,\qquad dt(E_2)=0.
\]

\medskip
\noindent\textbf{Step 2: Choose target orbits with admissible phases.}
To determine the required frequency before choosing the radius,
fix \(r\in I\) and a positive integer \(m\). Define the closed
parametrized orbit by
\[
 \gamma_{r,m}:S^1\longrightarrow P,\qquad
 \gamma_{r,m}([\lambda])=(r,[m\lambda]).
\]
In a local angular coordinate \(\lambda\),
\(d\gamma_{r,m}/d\lambda=mF(r)V\), so arc length satisfies
\(d\ell=mF(r)\,d\lambda\). By \eqref{eq:warp-connection},
\begin{equation}\label{eq:orbit-frame-derivatives}
 \nabla^{\gamma_{r,m}}_{\partial_\lambda}U=mF'(r)V,
 \qquad
 \nabla^{\gamma_{r,m}}_{\partial_\lambda}V=-mF'(r)U.
\end{equation}
Hence
\begin{equation}\label{eq:orbit-phase-derivative}
 \nabla^{\gamma_{r,m}}_{\partial_\lambda}
       \bigl(e^{\ii q\lambda}(U+\ii V)\bigr)
 =\ii\bigl(q-mF'(r)\bigr)e^{\ii q\lambda}(U+\ii V).
\end{equation}
Thus cancellation requires \(q=mF'(r)\), and the global spin
condition becomes \(mF'(r)\in\ZZ+\sigma_*\).
Choose \(r_*\) with \(F'(r_*)F''(r_*)\ne0\): near a point
where \(F''\ne0\), the function \(F'\) is strictly monotone
and has at most one zero. By \eqref{eq:warp-data}, this choice
ensures that both curvatures are nonzero.

Reverse \(r\) if necessary and choose a small interval \(J\)
about \(r_*\) where \(F'>0\), \(F''\ne0\), and
\(\epsilon=\operatorname{sign}K\) is constant.
The function \(F'|_J\) has a smooth inverse.
Choose \(q_m\in\ZZ+\sigma_*\) nearest to \(mF'(r_*)\); then
\[
 \left|\frac{q_m}{m}-F'(r_*)\right|\le\frac1{2m}.
\]
Since \(F'(J)\) is an open interval containing \(F'(r_*)>0\),
for all sufficiently large positive integers \(m\) we have
\(q_m>0\) and \(q_m/m\in F'(J)\). Thus
\begin{equation}\label{eq:warp-resonance}
 r_m=(F'|_J)^{-1}(q_m/m)
 \quad\hbox{satisfies}\quad
 r_m\longrightarrow r_*,\qquad mF'(r_m)=q_m.
\end{equation}
In this way a nearby orbit satisfies the phase condition exactly,
while retaining the two nonvanishing curvature properties.

\medskip
\noindent\textbf{Step 3: Construct global fields and solve the Dirac equation.}
Fix a sufficiently large positive integer \(m\), and use the radius
\(r_m\) and frequency \(q_m\) chosen in Step~2. Define the map
\begin{equation}\label{eq:C-map}
 \begin{aligned}
 \phi_m&:=\gamma_{r_m,m}\circ t:(M,g_0)\longrightarrow P,\\
 \phi_m([x,y])&=\bigl(r_m,[m(p_1x+p_2y)]\bigr).
 \end{aligned}
\end{equation}
The angular component has differential \(m\,dt\), and integrality
of \(m,p_1,p_2\) makes the map globally defined.
Put \(F_m=F(r_m)\), \(K_m=-F''(r_m)/F_m\), and
\(v_m=mF_m\rho>0\). Along \(\phi_m\), define the pullback target
frame explicitly by
\[
 U_m([x,y])=(\partial_r)_{\phi_m([x,y])},\qquad
 V_m([x,y])=F_m^{-1}(\partial_\varphi)_{\phi_m([x,y])}.
\]
These are global orthonormal sections of \(\phi_m^*TP\).

In the local parallel spinor frames choose the same constant unit column
\(\chi_\epsilon\) satisfying
\(E_1\cdot\chi_\epsilon=\ii\epsilon\chi_\epsilon\), where
\(\epsilon=\operatorname{sign}K_m\) was fixed in Step~2.
Such a choice is possible by the Clifford eigenspace construction
in \Cref{subsec:flat-spinors}, since \(E_1\) is parallel.
For a positive constant \(A_m\), to be determined in Step~4, set
\begin{equation}\label{eq:C-spinor}
 \xi_m=A_m e^{\ii q_m(p_1x+p_2y)}\chi_\epsilon,\qquad
 \Psi_m=\xi_m\otimes U_m+\ii\xi_m\otimes V_m.
\end{equation}
The congruences \eqref{eq:spin-congruences}, applied to
\(q_m\in\ZZ+\sigma_*\), give the coefficient of \(\xi_m\)
the prescribed spin signs in both source directions. Thus these
local expressions define \(\xi_m\in\Gamma(\Sigma M)\) and,
since \(U_m,V_m\) are global, \(\Psi_m\in\Gamma(E_{\phi_m})\).

The radial component of \(\phi_m\) is constant, and its angular
component has differential \(m\,dt\). The chain rule gives
\begin{equation}\label{eq:C-differential}
 \begin{aligned}
 d\phi_m(E_1)
 &=m\,dt(E_1)(\partial_\varphi)_{\phi_m}
   =mF_m\rho V_m=v_mV_m,\\
 d\phi_m(E_2)&=0.
 \end{aligned}
\end{equation}
Since \(A_m\) is constant and \(\chi_\epsilon\) is parallel,
pulling back \eqref{eq:orbit-phase-derivative} by \(t\) gives
\[
 \tnabla_X\Psi_m
 =\ii\bigl(q_m-mF'(r_m)\bigr)dt(X)\Psi_m=0
\]
for every source vector field \(X\), by \eqref{eq:warp-resonance}.
Hence \(\Psi_m\) is parallel and \(\D^{\phi_m}\Psi_m=0\).

\medskip
\noindent\textbf{Step 4: Choose the amplitude and verify the map equation.}
Let \(\kappa_m:=F'(r_m)/F_m\), the signed geodesic curvature
of the orbit \(r=r_m\) oriented by increasing \(\varphi\), as in
\eqref{eq:warp-data}. It satisfies
\(\nabla_{V_m}V_m=-\kappa_mU_m\) along that orbit.
The covering multiplicity does not change this curvature;
the speed of \(\phi_m\) in the unit source direction \(E_1\)
is \(v_m\). Thus \eqref{eq:C-differential} and the curve
calculation \eqref{eq:curve-balance} give
\begin{align}
 \tau_{g_0}(\phi_m)&=-v_m^2\kappa_mU_m
                  =-v_mq_m\rho U_m,\label{eq:C-tension}\\
 \cR(\phi_m,\Psi_m)&=-\epsilon v_mK_mA_m^2U_m.
 \label{eq:C-force}
\end{align}
Since \(\epsilon K_m=|K_m|\), choose
\begin{equation}\label{eq:C-amplitude}
 A_m=\left(\frac{q_m\rho}{|K_m|}\right)^{1/2}.
\end{equation}
This makes the two expressions equal. The map has rank one
because \(v_m>0\), and its tension is nowhere zero because
\(q_m\rho>0\).

Finally, \eqref{eq:C-differential} gives
\( |d\phi_m|^2=|d\phi_m(E_1)|^2+|d\phi_m(E_2)|^2=v_m^2 \).
Hence
\begin{equation}\label{eq:C-energy}
 E(\phi_m)=\frac12m^2F_m^2\rho^2\Area(M,g_0)\longrightarrow\infty,
\end{equation}
because \(F_m\to F(r_*)>0\).
Compose with the totally geodesic isometric immersion into \(N\), and replace
\(\Psi_m\) by \(e^{-u/2}\beta\Psi_m\) to pass from \(g_0\) to \(g\).
By \Cref{lem:conformal,lem:totally-geodesic}, this preserves the
equations, nonvanishing tension, and map energy, completing the proof.
\end{proof}

\subsection{Consequences and scope}
The annulus condition immediately applies to the usual nonflat
constant-curvature targets.
\begin{proof}[Proof of \Cref{cor:spaceforms}]
Fix \(p\in N\) and choose \(\varepsilon>0\) so that
\(\exp_p:B_\varepsilon(0)\subset T_pN\to B_\varepsilon(p)\)
is a diffeomorphism. The Jacobi equation gives the polar metric
\(dr^2+S_c(r)^2g_{S^{n-1}}\) on this normal ball, where
\[
 S_c(r)=
 \begin{cases}
  \sin(\sqrt c\,r)/\sqrt c,&c>0,\\
  \sinh(\sqrt{-c}\,r)/\sqrt{-c},&c<0.
 \end{cases}
\]
Choose any two-dimensional linear subspace \(\Pi\subset T_pN\)
and define the radial disk
\[
 S_\Pi:=\exp_p\bigl(\Pi\cap B_\varepsilon(0)\bigr).
\]
In normal coordinates, \(S_\Pi\) is the intersection of
\(\Pi\) with the coordinate ball.

For \(v\in\Pi\) and \(w\in\Pi^\perp\), define
\[
 A(v+w)=v-w,\qquad
 \mathscr R=\exp_p\circ A\circ\exp_p^{-1}.
\]
The orthogonal map \(A\) preserves the round metric of
\(S^{n-1}\). Since \(\mathscr R(r,\omega)=(r,A\omega)\),
the polar metric above satisfies \(\mathscr R^*h=h\), and
\[
 \operatorname{Fix}(\mathscr R)=S_\Pi.
\]
This fixed disk is totally geodesic. Indeed, if an ambient
geodesic starts at \(x\in S_\Pi\) with initial velocity in
\(T_xS_\Pi\), its image under \(\mathscr R\) has the same
initial data. Uniqueness of geodesics shows that it stays in
\(S_\Pi\) while it remains in the normal ball. For \(n=2\),
\(A=\id\) and \(S_\Pi=B_\varepsilon(p)\).

The induced metric on \(S_\Pi\setminus\{p\}\) is
\(dr^2+S_c(r)^2d\varphi^2\). Hence an annulus
\(a<r<b<\varepsilon\), with \(a>0\), has
\(F=S_c\) and \(F''=-cF\ne0\).
Apply \Cref{thm:warped-annulus}. All images lie in the chosen
normal ball, so the resulting maps are null-homotopic; rank,
nonvanishing tension, and energy growth follow from that theorem.
\end{proof}

Another closed target is the torus with metric
\[
 h=dr^2+(2+\sin r)^2d\varphi^2,\qquad
 (r,\varphi)\in(\RR/2\pi\ZZ)^2.
\]
Its curvature is \(K(r)=\sin r/(2+\sin r)\), and any small interval
in \(0<r<\pi/2\) supplies the required annulus.
A rotationally symmetric disk with nonzero curvature at its
center also contains a suitable annulus, since \(F''=-KF\).
Product and warped-product extensions are treated in
\Cref{sec:higher}.

The locality of the construction does not make the annulus
hypothesis automatic for every target. General geodesic polar
coordinates give \(dr^2+G(r,\varphi)^2d\varphi^2\); they need not
remove the angular dependence of \(G\).
Here \(\kappa=G_r/G\) and \(K=-G_{rr}/G\), so their ratio can vary
along a coordinate circle, preventing the constant-amplitude
balance \eqref{eq:curve-ratio}. In higher dimensions,
\(S_\Pi=\exp_p(\Pi\cap B_\varepsilon(0))\) contains the
radial geodesics through \(p\), but its second fundamental form
need not vanish. For a general metric, the corresponding
reflection \(\exp_p\circ A\circ\exp_p^{-1}\) need not be an
isometry.

\section{Removing the symmetry assumption}
\label{sec:surfaces}

A general coordinate circle need not have constant
\(\kappa/K\). By \eqref{eq:curve-balance}--\eqref{eq:curve-ratio},
we instead seek a smooth unit-speed closed immersed curve
\(\gamma:\RR/L\ZZ\to N\) whose curvature and total turning
\(\Theta_\gamma=\int_0^L\kappa\,d\ell\) satisfy
\begin{equation}\label{eq:Cprime-two-conditions}
 \kappa=aK\quad(a\ne0\text{ constant}),\qquad
 e^{\ii m_0\Theta_\gamma}=-1
 \quad\text{for some }m_0\in\mathbb{N}.
\end{equation}
Even and odd multiples of \(m_0\) give periodic and antiperiodic
phases, respectively; the source circle map will impose both spin signs.

The warped annulus suggests how to prescribe this phase without
symmetry. Recall the annulus \(P=I\times_F S^1\) in
\eqref{eq:warped-annulus}, where \(I\subset\RR\) is the
interval of radial coordinates. By \eqref{eq:warp-data}, the orbit
\(\gamma_r\), traversed once in increasing \(\varphi\), has
\(\Theta_{\gamma_r}=2\pi F'(r)\). Choose \(r_0,r\in I\)
with \(r_0<r\), and set
\(\mathcal A_{r_0,r}=[r_0,r]\times S^1\subset P\). Then
\begin{equation}\label{eq:warped-curvature-area}
 \int_{\mathcal A_{r_0,r}}K\,dA_h
 =-2\pi\int_{r_0}^r F''(s)\,ds
 =\Theta_{\gamma_{r_0}}-\Theta_{\gamma_r}.
\end{equation}
Thus curvature area controls the phase increment. We extend this
relation to general curves and use it to enforce
\eqref{eq:Cprime-two-conditions}.

\subsection{Curvature area and the total turning of a closed curve}
\label{subsec:curvature-area}
We work on an oriented coordinate disk \(D\) in the target surface,
with Gaussian curvature \(K\), rotation \(J\) by \(+\pi/2\),
and oriented area form \(dA_h\). Thus
\(dA_h(X,Y)=h(JX,Y)\).
For a smooth unit-speed closed immersed curve
\(\gamma:\RR/L\ZZ\to D\), we retain the convention of
\Cref{sec:annuli}:
\begin{equation}\label{eq:Cprime-Frenet}
 V=\gamma',\qquad U=-JV,\qquad
 \nabla_VV=-\kappa U=\kappa JV,\qquad
 \nabla_VU=\kappa V.
\end{equation}
Here primes denote arc-length derivatives, and \(L\) is the
length of one parametrized period.
By the Poincar\'e lemma \cite[Theorem~17.14]{Lee13}, choose a
smooth one-form \(\eta\) on \(D\) with \(d\eta=K\,dA_h\).
For any closed parametrized curve \(\gamma:[0,P]\to D\), set
\begin{equation}\label{eq:curvature-area}
 \mathcal A_K(\gamma)
 :=\int_\gamma\eta
 =\int_0^P\eta_{\gamma(s)}\bigl(\dot\gamma(s)\bigr)\,ds.
\end{equation}
The dot denotes differentiation in \(s\), which need not be arc
length. This \emph{signed curvature area} is the line-integral
term in the magnetic action \cite[Section~1.1]{CMP04}.
On \(D\), it is independent of the choice of primitive \(\eta\).
The functional is also defined on \(H^1\) loops.

To compute the total turning of the unit-speed curve \(\gamma\),
choose a smooth oriented orthonormal frame \(P_1,P_2\) on \(D\),
with \(P_2=JP_1\). Its connection one-form
\(\omega(X)=h(\nabla_XP_1,P_2)\) satisfies
\begin{equation}\label{eq:disk-connection-form}
 \nabla_XP_1=\omega(X)P_2,\qquad
 \nabla_XP_2=-\omega(X)P_1.
\end{equation}

The connection formulas give \(R(X,Y)P_1=d\omega(X,Y)P_2\).
Since our curvature convention gives
\(R(P_1,P_2)P_1=-KP_2\) and \(dA_h(P_1,P_2)=1\),
\begin{equation}\label{eq:disk-curvature-form}
 d\omega=-K\,dA_h.
\end{equation}

Choose a smooth real-valued angle \(\vartheta:[0,L]\to\RR\)
along one parametrized period such that
\[
 V(\ell)=\cos\vartheta(\ell)P_1(\gamma(\ell))
          +\sin\vartheta(\ell)P_2(\gamma(\ell)).
\]
Its endpoint increment is
\(\vartheta(L)-\vartheta(0)=2\pi n_\gamma\), where
\(n_\gamma\in\ZZ\) is the winding number of the unit tangent
relative to the disk frame.

Differentiating the expression for \(V\) and using
\eqref{eq:disk-connection-form} gives
\[
 \nabla_VV=(\vartheta'+\omega(V))JV,
 \qquad \kappa=\vartheta'+\omega(V),
\]
where the second identity follows from \eqref{eq:Cprime-Frenet}.

Equations \(d\eta=K\,dA_h\) and
\eqref{eq:disk-curvature-form} imply \(d(\omega+\eta)=0\).
Since the disk is contractible, this closed one-form is exact.
Its integral around \(\gamma\) is therefore zero, so
\(\int_\gamma\omega=-\int_\gamma\eta=-\mathcal A_K(\gamma)\).
Integrating the preceding expression for \(\kappa\) now gives
\begin{equation}\label{eq:turning-area}
 \begin{aligned}
 \Theta_\gamma:=\int_0^L\kappa\,d\ell
 &=\vartheta(L)-\vartheta(0)+\int_\gamma\omega\\
 &=2\pi n_\gamma-\mathcal A_K(\gamma).
 \end{aligned}
\end{equation}
In particular, the phase after one traversal is
\(e^{\ii\Theta_\gamma}=e^{-\ii\mathcal A_K(\gamma)}\).
For a positive integer \(m_0\), this yields
\begin{equation}\label{eq:area-for-spin}
 \mathcal A_K(\gamma)=\frac{\pi}{m_0}
 \quad\Longrightarrow\quad
 e^{\ii m_0\Theta_\gamma}
 =e^{2\pi\ii m_0n_\gamma}e^{-\ii\pi}=-1.
\end{equation}
For a smooth variation \(\Gamma_\varepsilon\) of \(\gamma\)
with field \(W\), differentiation in coordinates and periodic
integration by parts give
\[
 \left.\frac d{d\varepsilon}
       \mathcal A_K(\Gamma_\varepsilon)\right|_0
 =\int_0^L(\partial_j\eta_i-\partial_i\eta_j)W^jV^i\,d\ell
 =\int_0^L K(\gamma)dA_h(W,V)\,d\ell.
\]
Since \(U=-JV\) implies \(dA_h(U,V)=h(JU,V)=1\),
\begin{equation}\label{eq:normal-area-variation}
 \left.\frac d{d\varepsilon}
       \mathcal A_K(\Gamma_\varepsilon)\right|_0
 =\int_0^L K(\gamma)f\,d\ell\qquad(W=fU).
\end{equation}

\subsection{Short curves with constant curvature ratio}
\begin{lemma}\label{lem:short-magnetic-curves}
Let \((N,h)\) be a closed oriented Riemannian surface with
\(\mathcal K:=\max_N|K|>0\). There exists \(b_*>0\) such that for
each \(0<b<b_*\), there is a smooth closed immersed curve
\(\gamma_b\), contained in a coordinate disk, with
\[
 \mathcal A_K(\gamma_b)=b,\qquad
 \kappa_b=a_bK\circ\gamma_b,\quad a_b\ne0,
 \qquad |K\circ\gamma_b|>\frac{\mathcal K}{2}.
\]
Its length tends to zero as \(b\downarrow0\).
\end{lemma}

\begin{proof}
We minimize over \(H^1(\RR/2\pi\ZZ,N)\) loops, with energy
\begin{equation}\label{eq:loop-energy}
 \mathscr E(\gamma)=\frac12\int_0^{2\pi}|\dot\gamma(s)|_h^2\,ds.
\end{equation}
Here \(\dot\gamma=d\gamma/ds\); see \cite[Chapter~1]{Klingenberg78}
for the Sobolev loop-space framework.

Choose \(0<4r<\inj(N,h)\), so that every \(B(p_0,4r)\)
is a normal coordinate disk, and fix \(E_0>0\) with
\(\sqrt{4\pi E_0}<r\). By Cauchy--Schwarz,
\begin{equation}\label{eq:short-loop-length}
 L(\gamma)\le\sqrt{4\pi\mathscr E(\gamma)}<r
 \qquad\text{if }\mathscr E(\gamma)\le E_0.
\end{equation}
For each such loop, set \(p_\gamma:=\gamma(0)\). Its image lies
in \(B(p_\gamma,r)\). On \(D_\gamma:=B(p_\gamma,4r)\), choose
a primitive \(\eta\) of \(K\,dA_h\) and define
\(\mathcal A_K(\gamma)=\int_\gamma\eta\) as in
\eqref{eq:curvature-area}.

These local definitions are consistent when the initial point is
changed. Indeed, the radial contraction
\[
 (u,s)\longmapsto
 \exp_{p_\gamma}\bigl(u\exp_{p_\gamma}^{-1}(\gamma(s))\bigr),
 \qquad 0\le u\le1,
\]
gives a filling inside \(B(p_\gamma,r)\). If another point
\(p'=\gamma(s_0)\) is used as the center, its radial filling
lies in
\(B(p',r)\subset B(p_\gamma,2r)\subset D_\gamma\).
Both fillings can therefore be evaluated using the same primitive
on \(D_\gamma\); Stokes' theorem identifies both curvature
integrals with \(\int_\gamma\eta\).
All sufficiently small \(H^1\) variations of a fixed loop can
likewise be treated using a single primitive on \(D_\gamma\).
For \(H^1\) loops, the identities follow by smooth approximation
in the common coordinate chart.

\medskip
\noindent\textbf{Step 1: Construct competitors with small energy.}
Choose \(p\in N\) with \(|K(p)|=\mathcal K\), and oriented
normal coordinates \(z\) based on an orthonormal basis of
\(T_pN\). The standard normal-coordinate estimates are
\[
 h_{ij}(z)=\delta_{ij}+O(|z|^2),\qquad
 \sqrt{\det(h_{ij}(z))}=1+O(|z|^2).
\]
Let \(c_R^+\) be the positively oriented boundary of the small
geodesic ball \(B(p,R)\), represented by
\(z_R(\varphi)=(R\cos\varphi,R\sin\varphi)\).
Stokes' theorem, continuity of \(K\), and
\(|z_R'|_{\mathrm{eucl}}=R\) give
\begin{align*}
 \mathcal A_K(c_R^+)
 &=\int_{|z|<R}K(z)\sqrt{\det(h_{ij}(z))}\,dz^1dz^2
   =\pi K(p)R^2+o(R^2),\\
 L(c_R^+)&=\int_0^{2\pi}R(1+O(R^2))\,d\varphi
           =2\pi R+O(R^3).
\end{align*}
Orient \(c_R\) so that its curvature area is positive. Then
\[
 \mathcal A_K(c_R)=\pi\mathcal K R^2+o(R^2),\qquad
 L(c_R)=2\pi R+O(R^3).
\]
Since \(K\) has a fixed nonzero sign near \(p\),
\(\mathcal A_K(c_R)\) is continuous and
strictly increasing for small \(R>0\). Thus, for every sufficiently
small \(b>0\), there is \(R(b)\) with
\[
 \mathcal A_K(c_{R(b)})=b,\qquad
 R(b)^2=\frac{b}{\pi\mathcal K}+o(b).
\]
Parametrize these circles at constant speed on \(\RR/2\pi\ZZ\).
Then
\[
 \mathscr E(c_R)=\frac{L(c_R)^2}{4\pi}=\pi R^2+o(R^2),
 \qquad \mathscr E(c_{R(b)})=\frac b{\mathcal K}+o(b)<E_0
\]
for sufficiently small \(b\). The admissible class is therefore
nonempty, and comparison with this competitor yields
\begin{equation}\label{eq:loop-upper-bound}
 \begin{aligned}
 I(b)&:=\inf\{\mathscr E(\gamma):
       \mathscr E(\gamma)\le E_0,\ \mathcal A_K(\gamma)=b\}\\
 &\le\mathscr E(c_{R(b)})
 =\frac b{\mathcal K}+o(b)<E_0.
 \end{aligned}
\end{equation}
\medskip
\noindent\textbf{Step 2: Obtain a minimizer.}
We apply the direct method \cite[Section~1.1]{Dacorogna08}.
For sufficiently small \(b>0\), choose an admissible sequence
\((\gamma_j)\) with \(\mathscr E(\gamma_j)\to I(b)\).
By compactness of \(N\), pass to a subsequence for which
\(\gamma_j(0)\to p_\infty\). For all sufficiently large \(j\),
\(d_h(\gamma_j(0),p_\infty)<r/2\), and hence
\[
 d_h(\gamma_j(s),p_\infty)
 \le L(\gamma_j)+d_h(\gamma_j(0),p_\infty)
 <\frac{3r}{2}
 \quad\text{for all }s\in[0,2\pi].
\]
Choose oriented normal coordinates on \(B(p_\infty,4r)\), and
let \(z_j\) be the coordinate representations of these loops.
Uniform equivalence of the metric with the Euclidean metric on
\(\overline{B(p_\infty,2r)}\) gives an \(H^1\) bound on \(z_j\).
Weak compactness and the compact embedding \(H^1\hookrightarrow C^0\)
give, after passage to a subsequence,
\[
 z_j\rightharpoonup z\quad\text{in }H^1,
 \qquad z_j\longrightarrow z\quad\text{uniformly};
\]
see \cite[Theorems~3.18 and~8.8]{Brezis11}.
The limit represents an \(H^1\) loop \(\gamma\).

For a common primitive \(\eta=\eta_i(z)\,dz^i\) on
\(B(p_\infty,4r)\), we have
\begin{align*}
 \eta_i(z_j)&\longrightarrow\eta_i(z)
       &&\text{strongly in }L^2(0,2\pi),\\
 \dot z_j^i&\rightharpoonup\dot z^i
       &&\text{weakly in }L^2(0,2\pi).
\end{align*}
Strong--weak continuity of the \(L^2\) pairing
\cite[Proposition~3.5(iv)]{Brezis11} therefore gives
\[
 \mathcal A_K(\gamma_j)
 =\int_0^{2\pi}\eta_i(z_j)\dot z_j^i\,ds
 \longrightarrow\int_0^{2\pi}\eta_i(z)\dot z^i\,ds
 =\mathcal A_K(\gamma)=b.
\]

For energy, uniform convergence and the \(L^2\) derivative bound
imply
\[
 \left|\int_0^{2\pi}(h_{ik}(z_j)-h_{ik}(z))
                         \dot z_j^i\dot z_j^k\,ds\right|
 \le\|h(z_j)-h(z)\|_{L^\infty}\|\dot z_j\|_{L^2}^2\longrightarrow0.
\]
The metric \(h(z(s))\) defines an equivalent \(L^2\) Hilbert norm.
Its weak lower semicontinuity \cite[Proposition~3.5(iii)]{Brezis11}
therefore gives
\[
 \mathscr E(\gamma)
 =\frac12\int_0^{2\pi}h_{ik}(z)\dot z^i\dot z^k\,ds
 \le\liminf_j\mathscr E(\gamma_j)=I(b)<E_0.
\]
Thus \(\gamma_b:=\gamma\) is a minimizer with
\(\mathscr E(\gamma_b)=I(b)<E_0\), and is nonconstant since
\(\mathcal A_K(\gamma_b)=b>0\).

\medskip
\noindent\textbf{Step 3: Show that the minimizer stays where \(K\ne0\).}
Let \(b_j\downarrow0\), and choose corresponding minimizers
\(\gamma_j\). By \eqref{eq:loop-upper-bound} and
\eqref{eq:short-loop-length}, their lengths tend to zero.
Compactness of \(N\) allows us to pass to a subsequence with
\(\gamma_j(0)\to p_\infty\). Since
\[
 \sup_{s\in[0,2\pi]}d_h(\gamma_j(s),p_\infty)
 \le L(\gamma_j)+d_h(\gamma_j(0),p_\infty)\longrightarrow0,
\]
the curves converge uniformly to the constant curve \(p_\infty\).

Choose fixed oriented normal coordinates centered at \(p_\infty\),
and let \(z_j\) denote the coordinate representation of
\(\gamma_j\). Then \(z_j\to0\) uniformly. Define the coordinate
coefficient
\[
 \mu(z):=K(z)\sqrt{\det(h_{ab}(z))},\qquad
 K\,dA_h=\mu(z)\,dz^1\wedge dz^2.
\]
Since \(h_{ab}(0)=\delta_{ab}\), we have
\(\mu(0)=K(p_\infty)\).
To apply the mean-zero Poincar\'e inequality, set
\[
 \bar z_j=\frac1{2\pi}\int_0^{2\pi}z_j(s)\,ds,
 \qquad y_j(s)=z_j(s)-\bar z_j.
\]
Then \(\int_0^{2\pi}y_j(s)\,ds=0\) and \(\dot y_j=\dot z_j\).
Fix a coordinate rectangle \(Q\) about \(0\), containing
\(z_j(S^1)\) and \(\bar z_j\) for all sufficiently large \(j\).
On \(Q\), use the translated coordinates \(y=z-\bar z_j\)
and define
\[
 c_j(y^1,y^2)=\int_0^{y^1}
       \mu\bigl(\bar z_j+(\zeta,y^2)\bigr)\,d\zeta,
 \qquad \eta_j=c_j(y^1,y^2)\,dy^2.
\]
Then \(d\eta_j=K\,dA_h\). Since \(z_j,\bar z_j,y_j\to0\)
uniformly, continuity of \(\mu\) gives the uniform bound
\[
 |c_j(y_j(s))|
 \le\bigl(|K(p_\infty)|+o(1)\bigr)|y_j^1(s)|.
\]

The mean-zero periodic Poincar\'e inequality gives
\[
 \|y_j^1\|_{L^2(0,2\pi)}
 \le\|\dot y_j^1\|_{L^2(0,2\pi)}.
\]
Consequently,
\begin{align}
 b_j
 &=\left|\int_{\gamma_j}\eta_j\right|
  =\left|\int_0^{2\pi}c_j(y_j)\dot y_j^2\,ds\right|\notag\\
 &\le\bigl(|K(p_\infty)|+o(1)\bigr)
                  \|y_j^1\|_{L^2}\|\dot y_j^2\|_{L^2}\notag\\
 &\le\bigl(|K(p_\infty)|+o(1)\bigr)
                  \|\dot y_j^1\|_{L^2}\|\dot y_j^2\|_{L^2}\notag\\
 &\le\frac{|K(p_\infty)|+o(1)}2
       \left(\|\dot y_j^1\|_{L^2}^2+\|\dot y_j^2\|_{L^2}^2\right)
       \notag\\
 &=\bigl(|K(p_\infty)|+o(1)\bigr)\mathscr E(\gamma_j).
 \label{eq:concentration-estimate}
\end{align}
The last line uses \(\dot y_j=\dot z_j\) and
\(h_{ab}(z_j)=\delta_{ab}+o(1)\) uniformly along the curves.
Combining \eqref{eq:concentration-estimate} with
\eqref{eq:loop-upper-bound} and dividing by \(b_j>0\), we obtain
\[
 1\le\bigl(|K(p_\infty)|+o(1)\bigr)
                     \left(\frac1{\mathcal K}+o(1)\right).
\]
Passing to the limit gives \(1\le|K(p_\infty)|/\mathcal K\).
Since \(\mathcal K=\max_N|K|\), it follows that
\( |K(p_\infty)|=\mathcal K \).
Thus, for all sufficiently small \(b>0\), every minimizing
curve lies entirely where \(|K|>\mathcal K/2\).
In particular, \(K\) has a fixed nonzero sign along each such curve.

\medskip
\noindent\textbf{Step 4: Derive the equation and prove smoothness.}
Write \(\gamma=\gamma_b\), and use coordinates \(z\) on a disk
containing its image. A dot denotes differentiation with respect
to \(s\in\RR/2\pi\ZZ\).
Since \(\mathscr E(\gamma)=I(b)<E_0\) and energy is continuous
in \(H^1\), only the equality constraint \(\mathcal A_K=b\)
enters the multiplier argument.

A variation field along \(\gamma\) has coordinate expression
\[
 W(s)=w^i(s)(\partial_i)_{\gamma(s)},\qquad
 w\in H^1(\RR/2\pi\ZZ;\RR^2).
\]
Using \(z_\varepsilon=z+\varepsilon w\), the first-variation
computation in \Cref{subsec:curvature-area} gives
\begin{equation}\label{eq:area-first-variation}
 \begin{aligned}
 d\mathcal A_K|_\gamma(W)
 &=\int_0^{2\pi}K(\gamma)\,dA_h(W,\dot\gamma)\,ds\\
 &=-\int_0^{2\pi}K(\gamma)h(W,J\dot\gamma)\,ds.
 \end{aligned}
\end{equation}

To verify that this differential is nonzero, define the \(L^2\)
field \(Z=-K(\gamma)J\dot\gamma\).
Since \(\gamma\) is nonconstant and \(|K(\gamma)|>\mathcal K/2\),
\[
 \|Z\|_{L^2}^2
 =\int_0^{2\pi}K(\gamma)^2|\dot\gamma|_h^2\,ds
 \ge\frac{\mathcal K^2}4\int_0^{2\pi}|\dot\gamma|_h^2\,ds>0.
\]
By density, choose periodic \(H^1\) fields \(W_n\to Z\)
in \(L^2\). Then
\[
 d\mathcal A_K|_\gamma(W_n)
 =\int_0^{2\pi}h(Z,W_n)\,ds\longrightarrow\|Z\|_{L^2}^2>0.
\]
Thus \(d\mathcal A_K|_\gamma\ne0\), and the Lagrange multiplier
rule gives \(\lambda\in\RR\) with
\[
 d\mathscr E|_\gamma(W)
 =\lambda\,d\mathcal A_K|_\gamma(W)
 \quad\text{for every periodic }H^1\text{ field }W.
\]
In coordinates, differentiation of energy yields
\[
 d\mathscr E|_\gamma(W)
 =\int_0^{2\pi}
 \left(h_{ij}(z)\dot z^j\dot w^i
       +\frac12\partial_i h_{jk}(z)w^i\dot z^j\dot z^k\right)ds.
\]
Writing \(J(\partial_k)=J^j{}_k\partial_j\), the other side is
\[
 d\mathcal A_K|_\gamma(W)
 =-\int_0^{2\pi}K(z)h_{ij}(z)J^j{}_k(z)\dot z^k w^i\,ds.
\]
Integration by parts in the distributional sense gives
\begin{equation}\label{eq:weak-momentum}
 \frac d{ds}\bigl(h_{ij}(z)\dot z^j\bigr)
 =\frac12\partial_i h_{jk}(z)\dot z^j\dot z^k
       +\lambda K(z)h_{ij}(z)J^j{}_k(z)\dot z^k
\end{equation}
on the circle.

The right-hand side belongs to \(L^1\), so the momentum
\(h_{ij}(z)\dot z^j\) belongs to \(W^{1,1}\) and has an
absolutely continuous representative
\cite[Theorem~8.2]{Brezis11}. Inverting the continuous metric
shows that \(z\in C^1\). The right-hand side is then continuous,
giving \(z\in C^2\). Using
\[
 \frac d{ds}(h_{ij}\dot z^j)
 -\frac12\partial_i h_{jk}\dot z^j\dot z^k
 =h_{im}\bigl(\ddot z^m+\Gamma^m_{jk}\dot z^j\dot z^k\bigr),
\]
we rewrite \eqref{eq:weak-momentum} as
\begin{equation}\label{eq:magnetic-EL}
 \nabla_s\dot\gamma=\lambda K(\gamma)J\dot\gamma.
\end{equation}
Smoothness of the coefficients then gives
\(\gamma\in C^\infty(\RR/2\pi\ZZ,N)\).

Taking the inner product with \(\dot\gamma\) gives
\[
 \frac d{ds}\frac12|\dot\gamma|_h^2
 =h(\nabla_s\dot\gamma,\dot\gamma)
 =\lambda K(\gamma)h(J\dot\gamma,\dot\gamma)=0.
\]
Thus \(|\dot\gamma_b|\equiv v_0>0\). Set
\(V=v_0^{-1}\dot\gamma_b\). Then
\[
 \nabla_VV=v_0^{-2}\nabla_s\dot\gamma_b
 =\frac{\lambda}{v_0}K(\gamma_b)JV.
\]
Thus its signed geodesic curvature is
\(\kappa_b=a_bK\circ\gamma_b\), where \(a_b=\lambda/v_0\).
If \(\lambda=0\), \eqref{eq:magnetic-EL} would give a nonconstant
closed geodesic of length \(L<r<\inj(N,h)\), contradicting
injectivity of the exponential map. Hence \(a_b\ne0\).
Finally,
\[
 L^2=4\pi\mathscr E(\gamma_b)=4\pi I(b)
 \le4\pi\left(\frac b{\mathcal K}+o(b)\right)\longrightarrow0.
\]
This proves the lemma.
\end{proof}

\subsection{The spinor construction and the existence theorem}
We now return to the same source frame and the same cancellation
mechanism used in \Cref{thm:warped-annulus}. The only replacement is that the
orbit of a warped annulus is now a curve supplied by
\Cref{lem:short-magnetic-curves}. Its geodesic curvature may vary
along the curve; its ratio to \(K\) is constant.

\begin{proof}[Proof of \Cref{thm:nonflat-surface}]
First assume that \(N\) is oriented.

\medskip
\noindent\textbf{Step 1: Fix one curve and its phase.}
Choose a positive integer \(m_0\) so large that
\(b=\pi/m_0<b_*\), and take the curve of
\Cref{lem:short-magnetic-curves}. Reparametrize it by arc length:
\[
 \gamma:\RR/L\ZZ\longrightarrow N,\qquad
 \kappa=aK\circ\gamma,\quad a\ne0.
\]
Write
\begin{equation}\label{eq:Cprime-phase}
 \alpha(\ell)=\int_0^\ell\kappa(s)\,ds,
 \qquad \Theta_\gamma=\int_0^L\kappa\,d\ell,
 \qquad \alpha(\ell+L)=\alpha(\ell)+\Theta_\gamma.
\end{equation}
The functions \(\kappa\), \(V\), and \(U\) are periodic in
\(\ell\), whereas \(\alpha\) is defined on the real line.
By \eqref{eq:turning-area} and the choice of \(b\),
\begin{equation}\label{eq:Cprime-antiphase}
 e^{\ii m_0\Theta_\gamma}=e^{\ii m_0(2\pi n_\gamma-\pi/m_0)}=-1.
\end{equation}
The curve and \(m_0\) depend only on the target, and will remain
fixed as we vary the source data and the number of traversals.

\medskip
\noindent\textbf{Step 2: Specify the map and impose the spin signs.}
Use the flat representative \(g_0\) in
\eqref{eq:general-flat-torus}, and choose \((p,\sigma_*)\) by
\eqref{eq:choose-direction} for the prescribed spin structure
\((\sigma,\delta)\). Let \(t\), \(\rho\), and \(E_1,E_2\)
be as in \eqref{eq:circle-parameter} and
\eqref{eq:adapted-source-frame}.

Choose the numbers of traversals
\begin{equation}\label{eq:Cprime-coverings}
 m=
 \begin{cases}
  2m_0j,&\sigma_*=0,\quad j=1,2,\ldots,\\
  (2j+1)m_0,&\sigma_*=\tfrac12,\quad j=0,1,\ldots.
 \end{cases}
\end{equation}
Define the circle-valued map and its composition with the
closed target curve by
\begin{equation}\label{eq:Cprime-map}
 \begin{gathered}
 \ell_m:M\longrightarrow\RR/L\ZZ,\qquad
 \ell_m([x,y])=\left[\frac{mL}{2\pi}(p_1x+p_2y)\right]_L,\\
 \phi_m=\gamma\circ\ell_m:M\longrightarrow N,
 \qquad v_m=\frac{mL\rho}{2\pi}>0.
 \end{gathered}
\end{equation}
The integer \(m\) makes \(\ell_m\) well defined and records
the number of traversals of the target curve.
As in the warped-annulus construction, give the pulled-back target
frame a subscript \(m\):
\[
 U_m([x,y]):=U(\ell_m([x,y])),\qquad
 V_m([x,y]):=V(\ell_m([x,y])).
\]
These are unit fields along \(\phi_m\), with values in
\(T_{\phi_m([x,y])}N\). Also set
\(\kappa_m:=\kappa\circ\ell_m\) and \(K_m:=K\circ\phi_m\),
so \(\kappa_m=aK_m\). Unlike the orbit case, these curvature
functions may vary on the source. The chain rule gives
\begin{equation}\label{eq:Cprime-differential}
 d\phi_m(E_1)=v_mV_m,\qquad d\phi_m(E_2)=0.
\end{equation}
In the local parallel spinor frames choose the same constant unit column
\(\chi_\epsilon\) satisfying
\(E_1\cdot\chi_\epsilon=\ii\epsilon\chi_\epsilon\), where
\(\epsilon=\operatorname{sign}(a)\). Using local real angular
coordinates \(x,y\), define
\begin{equation}\label{eq:Cprime-spinor}
 \begin{gathered}
 A_m=\sqrt{v_m|a|},\qquad
 \xi_m=A_m\exp\left(\ii\alpha\left(
          \frac{mL}{2\pi}(p_1x+p_2y)\right)\right)\chi_\epsilon,\\
 \Psi_m=\xi_m\otimes(U_m+\ii V_m).
 \end{gathered}
\end{equation}
The argument of \(\alpha\) is the real representative of
\(\ell_m\) determined by these angular coordinates. The complex
target notation has the meaning fixed in \eqref{eq:complex-tensor}.

Changing the representative \(x^i\) by \(2\pi\) changes this
argument by \(mp_iL\), and hence multiplies the scalar
coefficient by \(e^{\ii mp_i\Theta_\gamma}\). By
\eqref{eq:Cprime-antiphase}--\eqref{eq:Cprime-coverings}, the
factor is \(+1\) for both directions in the trivial spin
structure. In the other three structures it is \((-1)^{p_i}\).
Thus, in every case,
\begin{equation}\label{eq:Cprime-descent}
 e^{\ii mp_1\Theta_\gamma}=e^{2\pi\ii\sigma},\qquad
 e^{\ii mp_2\Theta_\gamma}=e^{2\pi\ii\delta}.
\end{equation}
These are precisely the two transition rules from
\Cref{subsec:flat-spinors}. The local expressions therefore
define a global spinor \(\xi_m\) and, since \(U_m,V_m\) are
global, a section \(\Psi_m\in\Gamma(E_{\phi_m})\) on the
prescribed closed spin torus.

\medskip
\noindent\textbf{Step 3: Verify cancellation and curvature balance.}
The local curve calculation in \Cref{sec:annuli} applies with
speed \(v_m\), phase derivative \(\alpha'=\kappa\), and
constant amplitude \(A_m\). It gives
\(\tnabla\Psi_m=0\), hence \(\D^{\phi_m}\Psi_m=0\).
Since \(\epsilon A_m^2=v_ma\), \eqref{eq:curve-balance} gives
\[
 \cR_{g_0}(\phi_m,\Psi_m)
 =-\epsilon v_mK_mA_m^2U_m
 =-v_m^2aK_mU_m
 =-v_m^2\kappa_mU_m=\tau_{g_0}(\phi_m).
\]
The possibly variable curvatures cause no additional term:
their ratio \(a\), and therefore the amplitude, is constant.
Since \(v_m>0\), the map has rank one everywhere by
\eqref{eq:Cprime-differential}. Since \(\kappa_m\ne0\),
its tension is nowhere zero as well. Also,
\begin{equation}\label{eq:Cprime-energy}
 E(\phi_m)=\frac12v_m^2\Area(M,g_0)
 =\frac{m^2L^2\rho^2}{8\pi^2}\Area(M,g_0)
 \longrightarrow\infty.
\end{equation}
All maps take values in the same coordinate disk and are therefore
null-homotopic.

Finally, apply \Cref{lem:conformal} with the spinor
\(e^{-u/2}\beta\Psi_m\) to pass from \(g_0\) to \(g\).
This preserves the equations, nonvanishing tension, rank, and map
energy.

If \(N\) is nonorientable, let \(\pi:\widehat N\to N\)
be its orientation double cover; see
\cite[Proposition~15.40 and Theorem~15.41]{Lee13}.
The surface \(\widehat N\) is closed and oriented. With
\(\widehat h=\pi^*h\), the projection is a local isometry and
\(\widehat K=K\circ\pi\not\equiv0\).
Apply the orientable case to obtain
\((\widehat\phi_m,\widehat\Psi_m)\), and define
\[
 \phi_m=\pi\circ\widehat\phi_m,\qquad
 \Psi_m=(\id_{\Sigma M}\otimes d\pi)\widehat\Psi_m.
\]
A local isometry is a totally geodesic isometric immersion, so
\Cref{lem:totally-geodesic} preserves the equations and gives
\[
 \tau_g(\phi_m)=d\pi\bigl(\tau_g(\widehat\phi_m)\bigr).
\]
Since \(d\pi\) is an isometry, tension stays nowhere zero,
rank stays one, and energy is unchanged. Null-homotopy is
preserved by composition with \(\pi\).
\end{proof}

\section{Prescribed homotopy classes for negatively curved targets}
\label{sec:homotopy}

The curves used in \Cref{sec:surfaces} are contractible, so changing
their numbers of traversals preserves the null-homotopic map class.
For a noncontractible curve, repeated traversals generally change
that class. To prove \Cref{thm:negative-homotopy}, we instead keep
the circle map \(t\) fixed and deform the target curve within one
free homotopy class. The curvature-area constraint will prescribe
the phase after a single traversal. The local cancellation and
curvature balance remain those of \Cref{sec:annuli,sec:surfaces}.

Throughout this section, \((N,h)\) is closed and oriented, with
\(K<0\). We retain \(dA_h(X,Y)=h(JX,Y)\), \(U=-JV\), and
\(\nabla_VV=-\kappa U\). In particular,
\(k_0:=\min_N|K|>0\).

\subsection{The homotopy class and relative curvature area}
\label{subsec:relative-area}

Fix a non-null-homotopic class \(\mathfrak h\in[M,N]\), where
\([M,N]\) denotes homotopy classes without a fixed basepoint.
Use the angular coordinates of \eqref{eq:general-flat-torus},
choose a representative \(f:M\to N\), and set
\(o_M=[0,0]\) and \(o_N=f(o_M)\).
Let \(a_1,a_2\) be the classes of the coordinate loops
\(s\mapsto[s,0]\) and \(s\mapsto[0,s]\), respectively. Then
\(\pi_1(M,o_M)=\langle a_1,a_2\mid a_1a_2=a_2a_1\rangle
\cong\ZZ^2\); see \cite[Proposition~1.12, p.~34]{Hatcher02}.

The Cartan--Hadamard Theorem
\cite[Corollary~9.1.1., p.~337]{Jost26} makes \(N\) aspherical.
Thus the homomorphism \(f_*\), up to conjugacy, determines
the unbased map class
\cite[Propositions~1B.9 and~4A.2, pp.~90 and~422]{Hatcher02}.
Its image is nontrivial and abelian, hence infinite cyclic by
Preissmann's Theorem \cite[Corollary~13.7.1., p.~604]{Jost26}.
Choose a generator \(c\) of this image. There are unique integers
\(p_1,p_2\) with
\begin{equation}\label{eq:homotopy-data}
 \Gamma_f:=\operatorname{im}f_*=\langle c\rangle,\qquad
 f_*(a_i)=c^{p_i}\ (i=1,2),\qquad p=(p_1,p_2)\in\ZZ^2.
\end{equation}
Since the two images generate \(\Gamma_f\),
\(\gcd(|p_1|,|p_2|)=1\); thus \(p\) is primitive.
Replacing \(c\) by \(c^{-1}\) replaces \(p\) by \(-p\).

For \(t([x,y])=[p_1x+p_2y]_{2\pi}\), identify
\(\pi_1(S^1)\) with \(\ZZ\), so that \(t_*(a_i)=p_i\).
Let \(\mathfrak c\) be the free homotopy class represented by
\(c\). A based representative \(\gamma\) satisfies
\((\gamma\circ t)_*(a_i)=\gamma_*(p_i)=c^{p_i}=f_*(a_i)\).
Hence \(\gamma\circ t\simeq f\), and composing free homotopies
with \(t\) gives
\begin{equation}\label{eq:fixed-class-factorization}
 [\gamma\circ t]=\mathfrak h
 \quad\text{for every }\gamma:S^1\to N\text{ in }\mathfrak c.
\end{equation}
We keep \(p\) and \(t\) fixed and vary \(\gamma\) within
\(\mathfrak c\).

Every map from a closed oriented torus to \(N\) has degree zero:
this is immediate for null-homotopic maps, and otherwise
\eqref{eq:fixed-class-factorization} and \(H_2(S^1;\ZZ)=0\) give
\[
 \deg(f)[N]=f_*[M]=\gamma_*t_*[M]=0.
\]

Choose the closed geodesic \(\gamma_0:S^1\to N\) in
\(\mathfrak c\), with constant positive speed. Its existence
and uniqueness up to oriented reparametrization follow from
\cite[Theorem~9.1.1.,p.~336]{Jost26}.
The parametrization includes any multiplicity prescribed by
\(\mathfrak c\).
For a smooth loop \(\gamma\in\mathfrak c\), choose a smooth
free homotopy \(\mathcal H:[0,1]\times S^1\to N\) satisfying
\[
 \mathcal H(0,s)=\gamma_0(s),\qquad
 \mathcal H(1,s)=\gamma(s).
\]
Orient the cylinder by \(d\varepsilon\wedge ds\), so that its
oriented boundary maps to \(\gamma-\gamma_0\), and define
\begin{equation}\label{eq:relative-curvature-area}
 \begin{aligned}
 \mathcal A_{K,\gamma_0}(\gamma)
 &:=\int_{[0,1]\times S^1}\mathcal H^*(K\,dA_h)\\
 &=\int_0^1\int_0^{2\pi}
 K(\mathcal H(\varepsilon,s))\,
 dA_h(\partial_\varepsilon\mathcal H,\partial_s\mathcal H)
 \,ds\,d\varepsilon.
 \end{aligned}
\end{equation}
Given two such homotopies \(\mathcal H_1,\mathcal H_2\),
reparametrize them in \(\varepsilon\) to be independent of
\(\varepsilon\) near \(0\) and \(1\); change of variables
preserves their integrals. Gluing \(\mathcal H_1\) to
\(\mathcal H_2(1-\varepsilon,s)\) gives a smooth map
\(F:T^2\to N\). Hence
\[
 \int_{[0,1]\times S^1}\mathcal H_1^*(K\,dA_h)
 -\int_{[0,1]\times S^1}\mathcal H_2^*(K\,dA_h)
 =\int_{T^2}F^*(K\,dA_h)
 =\deg(F)\int_NK\,dA_h=0.
\]
Thus \(\mathcal A_{K,\gamma_0}\) is well defined and vanishes
at \(\gamma_0\). Orientation-preserving reparametrizations of
\(\gamma\) leave it unchanged: the corresponding homotopies
have rank at most one and contribute zero curvature area.

Consider the Sobolev loop space
\[
 \Lambda_{\mathfrak c}N
 :=\{\gamma\in H^1(S^1,N):[\gamma]=\mathfrak c\},
 \qquad S^1=\RR/2\pi\ZZ.
\]
It carries the usual Hilbert manifold structure, with tangent space
\(H^1(S^1,\gamma^*TN)\) and charts given by pointwise exponential
maps \cite[Chapter~1]{Klingenberg78}. The embedding
\(H^1\hookrightarrow C^0\) makes free homotopy classes locally
constant.

\begin{lemma}[Relative area on the Sobolev loop space]
\label{lem:relative-area}
The functional \eqref{eq:relative-curvature-area} has a unique
continuous extension to \(\Lambda_{\mathfrak c}N\). This
extension is continuously differentiable, with
\begin{equation}\label{eq:relative-area-variation}
 d\mathcal A_{K,\gamma_0}|_\gamma(\mathbf X)
 =\int_0^{2\pi}K(\gamma)\,dA_h(\mathbf X,\dot\gamma)\,ds
 \qquad
 \bigl(\mathbf X\in H^1(S^1,\gamma^*TN)\bigr).
\end{equation}
Moreover, if \(\gamma_j\to\gamma\) uniformly and
\(\gamma_j\rightharpoonup\gamma\) weakly in \(H^1\), then
\(\mathcal A_{K,\gamma_0}(\gamma_j)
\to\mathcal A_{K,\gamma_0}(\gamma)\).
Here weak convergence can be checked in target charts on
finitely many subintervals of the parameter circle.
\end{lemma}

\begin{proof}
For continuous loops, write
\(d_\infty(\gamma_1,\gamma_2)=\max_s d_h(\gamma_1(s),\gamma_2(s))\).
If smooth loops \(\gamma_1,\gamma_2\in\mathfrak c\) satisfy
\(d_\infty(\gamma_1,\gamma_2)<\inj(N,h)/2\), define
\[
 \mathcal H(\varepsilon,s)
 =\exp_{\gamma_1(s)}\bigl(
       \varepsilon\exp_{\gamma_1(s)}^{-1}\gamma_2(s)\bigr).
\]
Homotopy independence gives
\begin{equation}\label{eq:relative-area-local}
 \mathcal A_{K,\gamma_0}(\gamma_2)
 -\mathcal A_{K,\gamma_0}(\gamma_1)
 =\int_0^1\int_0^{2\pi}
 K(\mathcal H)dA_h(\partial_\varepsilon\mathcal H,
                         \partial_s\mathcal H)\,ds\,d\varepsilon.
\end{equation}
Smooth dependence of the short geodesic on its endpoints and
compactness of \(N\) yield
\[
 |\partial_\varepsilon\mathcal H|_h
 \le d_\infty(\gamma_1,\gamma_2),\qquad
 |\partial_s\mathcal H|_h
 \le C_h\bigl(|\dot\gamma_1|_h+|\dot\gamma_2|_h\bigr).
\]
Consequently,
\begin{equation}\label{eq:relative-area-estimate}
 \left|\mathcal A_{K,\gamma_0}(\gamma_2)
            -\mathcal A_{K,\gamma_0}(\gamma_1)\right|
 \le C\,d_\infty(\gamma_1,\gamma_2)
                 \bigl(L(\gamma_1)+L(\gamma_2)\bigr),
\end{equation}
where \(C=C_h\|K\|_{L^\infty(N)}\).

For \(\gamma\in\Lambda_{\mathfrak c}N\), choose smooth
\(\gamma_j\to\gamma\) strongly in \(H^1\)
\cite[Chapter~1]{Klingenberg78}. The convergence is uniform,
so \(\gamma_j\in\mathfrak c\) for large \(j\), and their
lengths are bounded. Estimate \eqref{eq:relative-area-estimate}
therefore defines \(\mathcal A_{K,\gamma_0}(\gamma)\) as the
limit of their area values, independently of the approximation.
On finitely many common coordinate intervals, strong \(H^1\)
convergence gives \(L^2\) convergence of speeds and hence convergence
of lengths. Thus \eqref{eq:relative-area-estimate} extends to
\(H^1\) loops. In particular,
\[
 d_\infty(\gamma_j,\gamma)\longrightarrow0,
 \quad\sup_jL(\gamma_j)<\infty
 \quad\Longrightarrow\quad
 \mathcal A_{K,\gamma_0}(\gamma_j)
       \longrightarrow\mathcal A_{K,\gamma_0}(\gamma).
\]
This proves the claimed continuity and sequential convergence;
density gives uniqueness of the extension.

To prove continuous differentiability, fix
\(\gamma\in\Lambda_{\mathfrak c}N\) and use the local chart
\[
 \mathcal F_\gamma(\mathbf X)(s)=\exp_{\gamma(s)}\mathbf X(s),
 \qquad
 \mathcal H_{\mathbf X}(\varepsilon,s)
      =\exp_{\gamma(s)}(\varepsilon\mathbf X(s)),
\]
for sufficiently small \(\mathbf X\in H^1(S^1,\gamma^*TN)\).
The local homotopy formula is
\begin{equation}\label{eq:relative-area-strip}
 \begin{aligned}
 &\mathcal A_{K,\gamma_0}(\mathcal F_\gamma(\mathbf X))
       -\mathcal A_{K,\gamma_0}(\gamma)\\
 &\qquad=\int_0^1\int_0^{2\pi}
 K(\mathcal H_{\mathbf X})
 dA_h(\partial_\varepsilon\mathcal H_{\mathbf X},
                 \partial_s\mathcal H_{\mathbf X})
       \,ds\,d\varepsilon.
 \end{aligned}
\end{equation}
For \(H^1\) data, this follows by smooth approximation in local
bundle coordinates: the homotopies and their
\(\varepsilon\)-derivatives converge uniformly, and their
\(s\)-derivatives converge in \(L^2\).
In these coordinates, smoothness of the exponential map and
\(H^1\hookrightarrow C^0\) show that
\(\mathbf X\mapsto(\mathcal H_{\mathbf X},
\partial_\varepsilon\mathcal H_{\mathbf X})\) is \(C^1\) into
\(C^0\), while \(\mathbf X\mapsto\partial_s\mathcal H_{\mathbf X}\)
is \(C^1\) into \(L^2\). Thus the integral in
\eqref{eq:relative-area-strip} is \(C^1\).
At \(\mathbf X=0\), we have
\(\partial_\varepsilon\mathcal H_0=0\) and
\[
 \left.\frac d{du}\right|_{u=0}
       \partial_\varepsilon\mathcal H_{u\mathbf Y}
 =\mathbf Y,\qquad
 \partial_s\mathcal H_0=\dot\gamma.
\]
Differentiating \eqref{eq:relative-area-strip} therefore gives
\[
 d\mathcal A_{K,\gamma_0}|_\gamma(\mathbf Y)
 =\int_0^{2\pi}K(\gamma)dA_h(\mathbf Y,\dot\gamma)\,ds,
\]
as required.
\end{proof}

\begin{lemma}[Relative turning identity]
\label{lem:relative-turning}
Let \(\gamma:\RR/2\pi\ZZ\to N\) be a smooth immersed
loop in \(\mathfrak c\), with signed geodesic curvature
\(\kappa\). In its given parameter \(s\), set
\[
 \Theta_\gamma:=\int_\gamma\kappa\,d\ell
 =\int_0^{2\pi}\kappa(s)|\dot\gamma(s)|_h\,ds.
\]
Then
\begin{equation}\label{eq:relative-turning}
 \Theta_\gamma\equiv-\mathcal A_{K,\gamma_0}(\gamma)
                  \pmod{2\pi},
 \qquad
 e^{\ii\Theta_\gamma}
 =e^{-\ii\mathcal A_{K,\gamma_0}(\gamma)}.
\end{equation}
\end{lemma}

\begin{proof}
Use a smooth homotopy \(\mathcal H\) as in
\eqref{eq:relative-curvature-area}. Along \(\gamma_0\), let
\(V_0=\dot\gamma_0/|\dot\gamma_0|\) and \(U_0=-JV_0\).
These fields form a periodic parallel oriented frame. Extend it
to an oriented orthonormal frame
\(P_1,P_2\) of \(\mathcal H^*TN\) by parallel transport along the
paths \(\varepsilon\mapsto\mathcal H(\varepsilon,s)\), taking
\((P_1,P_2)|_{\varepsilon=0}=(U_0,V_0)\).
Periodicity in \(s\) follows from that of \(\mathcal H\) and
the initial frame, by uniqueness of parallel transport.

Let \(\omega\) be its connection one-form on the cylinder:
\(\nabla P_1=\omega P_2\), \(\nabla P_2=-\omega P_1\).
The same curvature computation as
\eqref{eq:disk-curvature-form} gives
\[
 d\omega=-\mathcal H^*(K\,dA_h),\qquad
 \omega(\partial_\varepsilon)=0,\qquad
 \omega(\partial_s)|_{\varepsilon=0}=0.
\]
Stokes' theorem therefore yields
\[
 \int_{\{1\}\times S^1}\omega
 =-\mathcal A_{K,\gamma_0}(\gamma).
\]
Along the upper boundary write
\(V=\dot\gamma/|\dot\gamma|
=\cos\vartheta\,P_1+\sin\vartheta\,P_2\), with a real angle
\(\vartheta\) on \([0,2\pi]\). Since \(d\ell=|\dot\gamma|\,ds\),
the Frenet and connection formulas give
\[
 \kappa|\dot\gamma|JV=\nabla_sV
 =\left(\frac{d\vartheta}{ds}
         +\omega(\partial_s)|_{\varepsilon=1}\right)JV.
\]
Periodicity of \(V\) and the frame implies
\(\vartheta(2\pi)-\vartheta(0)=2\pi n\), \(n\in\ZZ\).
Integration yields
\(\Theta_\gamma=2\pi n-\mathcal A_{K,\gamma_0}(\gamma)\),
proving \eqref{eq:relative-turning}.
\end{proof}

\subsection{Constrained curves in a nontrivial free homotopy class}
\label{subsec:class-minimization}

We minimize the energy \(\mathscr E\) from \eqref{eq:loop-energy}
on \(\Lambda_{\mathfrak c}N\) at fixed relative curvature area.

\begin{lemma}\label{lem:class-magnetic-curves}
For every \(b\in\RR\), the infimum
\begin{equation}\label{eq:class-infimum}
 I_{\mathfrak c}(b)
 :=\inf\{\mathscr E(\gamma):\gamma\in\Lambda_{\mathfrak c}N,
                \ \mathcal A_{K,\gamma_0}(\gamma)=b\}
\end{equation}
is finite and attained. If \(b\ne0\), every minimizer
\(\gamma_b:S^1\to N\) is smooth, has constant positive
speed \(v_b\), and satisfies
\begin{equation}\label{eq:class-magnetic-curve}
 \nabla_s\dot\gamma_b=\lambda_b K(\gamma_b)J\dot\gamma_b,
 \qquad \lambda_b\ne0,
 \qquad \kappa_b=a_bK\circ\gamma_b,\quad
 a_b=\frac{\lambda_b}{v_b}\ne0.
\end{equation}
Here the dot denotes differentiation in the \(2\pi\)-periodic
parameter \(s\), and \(\kappa_b\) is signed geodesic curvature,
computed using \(d\ell=v_b\,ds\). In particular,
\(\kappa_b\) is nowhere zero. Moreover,
\begin{equation}\label{eq:class-energy-growth}
 I_{\mathfrak c}(b)\longrightarrow\infty
 \quad\text{as }|b|\longrightarrow\infty.
\end{equation}
\end{lemma}

\begin{proof}
\noindent\textbf{Step 1: Produce competitors and obtain a minimizer.}
For \(b=0\), the reference curve \(\gamma_0\) is admissible.
Suppose \(b\ne0\). Choose oriented normal coordinates \(z\)
centered at \(\gamma_0(0)\) on a disk \(D_{R_0}\), and set
\(D_R:=\{|z|<R\}\) for \(0<R<R_0\).
Let \(\beta_R\) be the radial path with
\(z(\beta_R(\varepsilon))=(\varepsilon R,0)\),
\(0\le\varepsilon\le1\). Let \(c_R\) traverse \(\partial D_R\)
once from \(\beta_R(1)\), clockwise if \(b>0\) and
counterclockwise if \(b<0\). For a positive integer \(n_0\), set
\[
 \widehat\gamma_{R,n_0}
 :=\gamma_0*\beta_R*c_R^{\,n_0}*\beta_R^{-1},
\]
where concatenation is read from left to right and
\(\beta_R^{-1}\) is the reversed path. The inserted loop is
contractible, so \(\widehat\gamma_{R,n_0}\in\mathfrak c\).

Choose \(\eta\) on \(D_{R_0}\) with \(d\eta=K\,dA_h\).
Stokes' theorem applied piecewise gives
\[
 \begin{aligned}
 \mathcal A_{K,\gamma_0}(\widehat\gamma_{R,n_0})
 &=\int_{\beta_R}\eta+n_0\int_{c_R}\eta
                       +\int_{\beta_R^{-1}}\eta\\
 &=n_0\int_{c_R}\eta
 =\operatorname{sign}(b)\,n_0\int_{D_R}|K|\,dA_h.
 \end{aligned}
\]
Since \(K<0\), the function \(R\mapsto\int_{D_R}|K|\,dA_h\)
is continuous, strictly increasing, and tends to zero as
\(R\downarrow0\). Choose \(n_0\) sufficiently large and then
\(R\) such that
\[
 \int_{D_R}|K|\,dA_h=\frac{|b|}{n_0}.
\]
Parametrize the joined curve at constant speed on
\([0,2\pi]\). It is Lipschitz, with
\[
 \widehat L=L(\gamma_0)+2L(\beta_R)+n_0L(c_R),\qquad
 \mathscr E(\widehat\gamma_{R,n_0})
 =\frac{\widehat L^2}{4\pi}<\infty.
\]
Thus the constraint admits a finite-energy competitor and
\(I_{\mathfrak c}(b)<\infty\).

Choose an admissible minimizing sequence \((\gamma_j)\) with
\(\mathscr E(\gamma_j)\le C:=I_{\mathfrak c}(b)+1\).
For \(0\le s<t\le2\pi\), Cauchy--Schwarz gives
\[
 d_h(\gamma_j(s),\gamma_j(t))
 \le\int_s^t|\dot\gamma_j(u)|_h\,du
 \le\sqrt{2C}\,|t-s|^{1/2}.
\]
Compactness of \(N\) and Arzel\`a--Ascoli yield a uniformly
convergent subsequence \(\gamma_j\to\gamma\).
Choose finitely many parameter intervals whose images under
\(\gamma\) are compactly contained in target charts. Uniform
convergence places \(\gamma_j\) in the same charts for large \(j\).
The energy bound gives \(H^1\) bounds in these coordinates;
after a further subsequence, \(\gamma\in H^1(S^1,N)\) and
\(\gamma_j\rightharpoonup\gamma\) weakly in \(H^1\).
Short pointwise geodesics between uniformly close loops show
that \(\gamma\in\Lambda_{\mathfrak c}N\).
By \Cref{lem:relative-area} and weak lower semicontinuity of energy,
\[
 \mathcal A_{K,\gamma_0}(\gamma)=b,\qquad
 \mathscr E(\gamma)\le\liminf_j\mathscr E(\gamma_j)
                    =I_{\mathfrak c}(b).
\]
Hence \(\gamma\) attains the infimum.

\medskip
\noindent\textbf{Step 2: Derive the equation and exclude a zero multiplier.}
Let \(\gamma=\gamma_b\) be a minimizer with \(b\ne0\).
By \eqref{eq:relative-area-variation},
\[
 d\mathcal A_{K,\gamma_0}|_\gamma(W)
 =\int_0^{2\pi}h(Z,W)\,ds,\qquad
 Z=-K(\gamma)J\dot\gamma.
\]
Since \(\mathfrak c\) is nontrivial and \(|K|\ge k_0>0\),
\[
 \|Z\|_{L^2}^2
 \ge k_0^2\int_0^{2\pi}|\dot\gamma|_h^2\,ds>0.
\]
Density of \(H^1\) sections in \(L^2\) therefore gives
\(d\mathcal A_{K,\gamma_0}|_\gamma\ne0\).
Since \(\Lambda_{\mathfrak c}N\) is open in \(H^1(S^1,N)\),
the Lagrange multiplier rule yields \(\lambda_b\in\RR\) with
\[
 d\mathscr E|_\gamma(W)
 =\lambda_b\,d\mathcal A_{K,\gamma_0}|_\gamma(W)
 \quad\text{for every periodic }H^1\text{ field }W.
\]
Testing on each coordinate interval gives
\eqref{eq:weak-momentum}, with \(\lambda=\lambda_b\).
The regularity and constant-speed arguments of Step~4 of
\Cref{lem:short-magnetic-curves} give smoothness, constant positive
speed, and the differential equation and curvature identity in
\eqref{eq:class-magnetic-curve}.
If \(\lambda_b=0\), uniqueness of the closed geodesic in
\(\mathfrak c\) would make \(\gamma_b\) an oriented
reparametrization of \(\gamma_0\), giving
\[
 b=\mathcal A_{K,\gamma_0}(\gamma_b)
   =\mathcal A_{K,\gamma_0}(\gamma_0)=0,
\]
a contradiction. Hence
\[
 a_b:=\frac{\lambda_b}{v_b}\ne0,\qquad
 |\kappa_b|=|a_b|\,|K\circ\gamma_b|
 \ge |a_b|k_0>0.
\]

\medskip
\noindent\textbf{Step 3: Prove the energy growth.}
If \(|b_j|\to\infty\) and the minimizers \(\gamma_{b_j}\) had
uniformly bounded energy, Step~1 would give a subsequence converging
uniformly and weakly in
\(H^1\) to \(\gamma\in\Lambda_{\mathfrak c}N\).
By \Cref{lem:relative-area},
\[
 b_j=\mathcal A_{K,\gamma_0}(\gamma_{b_j})
 \longrightarrow\mathcal A_{K,\gamma_0}(\gamma)\in\RR,
\]
a contradiction. For \(b\ne0\), constant speed also gives
\begin{equation}\label{eq:class-length-energy}
 L_b=2\pi v_b,\qquad
 I_{\mathfrak c}(b)=\pi v_b^2=\frac{L_b^2}{4\pi}.
\end{equation}
\end{proof}

\subsection{Spin signs and the construction in the prescribed class}
\label{subsec:class-spinors}

The prescribed map class fixes the primitive direction
\(p\) in \eqref{eq:homotopy-data}. For one traversal of the
target curve, the phase increments along the two source
generators are \(p_1\Theta_\gamma\) and \(p_2\Theta_\gamma\).
The spin transition rules \eqref{eq:spin-x} and
\eqref{eq:spin-y} therefore require
\begin{equation}\label{eq:class-phase-signs}
 e^{\ii p_1\Theta_\gamma}=e^{2\pi\ii\sigma},\qquad
 e^{\ii p_2\Theta_\gamma}=e^{2\pi\ii\delta}.
\end{equation}
Consequently,
\[
 e^{2\pi\ii(\sigma p_2-\delta p_1)}
 =\frac{(e^{\ii p_1\Theta_\gamma})^{p_2}}
        {(e^{\ii p_2\Theta_\gamma})^{p_1}}=1,
\]
so the necessary compatibility condition is
\begin{equation}\label{eq:class-spin-compatibility}
 \sigma p_2-\delta p_1\in\ZZ.
\end{equation}
Conversely, \(\gcd(|p_1|,|p_2|)=1\) gives
\(p\not\equiv(0,0)\pmod2\), and
\eqref{eq:class-spin-compatibility} is equivalent to
\((2\sigma,2\delta)\equiv(0,0)\) or \(p\pmod2\).
Thus there is a unique \(\sigma_*\in\{0,\tfrac12\}\) such that
\begin{equation}\label{eq:class-spin-direction}
 (\sigma,\delta)\equiv\sigma_*p\pmod{\ZZ^2}.
\end{equation}
It therefore suffices to construct a curve satisfying
\(e^{\ii\Theta_\gamma}=e^{2\pi\ii\sigma_*}\).

\begin{proposition}[Compatible spin structures]
\label{prop:compatible-homotopy}
Let \((N,h)\) be a closed oriented surface with \(K<0\),
and let \((M,g)\) be a closed oriented torus with any fixed
spin structure. For a non-null-homotopic class
\(\mathfrak h\in[M,N]\), choose its primitive direction
\(p\) as in \eqref{eq:homotopy-data}. If
\eqref{eq:class-spin-compatibility} holds, then \(\mathfrak h\)
contains a sequence of coupled Dirac-harmonic maps satisfying
\eqref{eq:main-conclusions}, with parallel twisted spinors
for a flat conformal representative.
\end{proposition}

\begin{proof}[Proof of \Cref{prop:compatible-homotopy,thm:negative-homotopy}]
The null-homotopic class is covered by \Cref{thm:nonflat-surface}.
Fix a non-null-homotopic class with data \(\mathfrak c,p\)
from \eqref{eq:homotopy-data}, and let
\(\sigma_*\) satisfy \eqref{eq:class-spin-direction}.
Work first with \(g_0\) in \eqref{eq:general-flat-torus}.

\medskip
\noindent\textbf{Step 1: Choose the curves and determine the phase.}
Choose
\begin{equation}\label{eq:class-area-sequence}
 b_j=2\pi(j+\sigma_*),\qquad j=1,2,\ldots.
\end{equation}
Let \(\gamma_{b_j}\) be the \(2\pi\)-periodic minimizers from
\Cref{lem:class-magnetic-curves}, with speeds \(v_{b_j}>0\)
and multipliers \(\lambda_{b_j}\ne0\). Define their arc-length
parametrizations by
\begin{equation}\label{eq:class-reparametrization}
 \begin{gathered}
 L_j=2\pi v_{b_j},\qquad \ell=v_{b_j}s,\\
 \gamma_j:\RR/L_j\ZZ\longrightarrow N,\qquad
 \gamma_j(\ell)=\gamma_{b_j}\left(\frac{\ell}{v_{b_j}}\right).
 \end{gathered}
\end{equation}
Set \(V_j=\gamma_j'\) and \(U_j=-JV_j\), with primes denoting
\(\ell\)-derivatives. The parameter relations in
\Cref{sec:annuli} and \Cref{lem:class-magnetic-curves} give
\begin{equation}\label{eq:class-curvature-conversion}
 a_j:=\frac{\lambda_{b_j}}{v_{b_j}}\ne0,\qquad
 \kappa_j(\ell)=a_jK(\gamma_j(\ell)),\qquad
 \kappa_j(v_{b_j}s)=\kappa_{b_j}(s).
\end{equation}
The total turning is unchanged:
\begin{equation}\label{eq:class-turning-conversion}
 \Theta_{\gamma_j}
 =\int_0^{L_j}\kappa_j(\ell)\,d\ell
 =\int_0^{2\pi}\kappa_{b_j}(s)v_{b_j}\,ds
 =\Theta_{\gamma_{b_j}}.
\end{equation}
Since \(\mathcal A_{K,\gamma_0}(\gamma_{b_j})=b_j\),
\Cref{lem:relative-turning} yields
\begin{equation}\label{eq:class-selected-phase}
 e^{\ii\Theta_{\gamma_j}}=e^{-\ii b_j}
 =e^{2\pi\ii\sigma_*},
\end{equation}
where \(\sigma_*\in\{0,\tfrac12\}\).

\medskip
\noindent\textbf{Step 2: Construct the map and close the spinor.}
For each \(j\), use the local formulas in Step~2 of the proof of
\Cref{thm:nonflat-surface} with \(\gamma=\gamma_j\),
\(a=a_j\), and \(m=1\), retaining the prescribed direction \(p\).
With \(t,\rho,E_1,E_2\) as in \eqref{eq:circle-parameter}
and \eqref{eq:adapted-source-frame}, set
\begin{equation}\label{eq:class-map}
 \begin{gathered}
 \ell_j([x,y])=[v_{b_j}(p_1x+p_2y)]_{L_j},\qquad
 v_j=\rho v_{b_j}=\frac{L_j\rho}{2\pi}>0,\\
 \phi_j:=\gamma_j\circ\ell_j=\gamma_{b_j}\circ t.
 \end{gathered}
\end{equation}
Since \(\gamma_{b_j}\in\Lambda_{\mathfrak c}N\),
\eqref{eq:fixed-class-factorization} gives
\([\phi_j]=\mathfrak h\) for every \(j\).
Set \(V_j^\phi=V_j\circ\ell_j\) and
\(U_j^\phi=U_j\circ\ell_j\). Then
\begin{equation}\label{eq:class-map-differential}
 d\phi_j(E_1)=v_jV_j^\phi,\qquad d\phi_j(E_2)=0.
\end{equation}
Regard \(\kappa_j\) as an \(L_j\)-periodic function on \(\RR\),
and set
\[
 \alpha_j(\ell)=\int_0^\ell\kappa_j(\zeta)\,d\zeta,
 \qquad
 \alpha_j(\ell+L_j)=\alpha_j(\ell)+\Theta_{\gamma_j}.
\]
Choose
\[
 A_j=\sqrt{v_j|a_j|}=\sqrt{\rho|\lambda_{b_j}|},\qquad
 \epsilon_j=\operatorname{sign}(a_j),
\]
and a local parallel unit spinor \(\chi_{\epsilon_j}\) with
\(E_1\cdot\chi_{\epsilon_j}=\ii\epsilon_j\chi_{\epsilon_j}\).
In local real angular coordinates, define
\begin{equation}\label{eq:class-spinor}
 \xi_j=A_j e^{\ii\alpha_j(v_{b_j}(p_1x+p_2y))}
           \chi_{\epsilon_j},\qquad
 \Psi_j=\xi_j\otimes(U_j^\phi+\ii V_j^\phi).
\end{equation}
The argument of \(\alpha_j\) increases by \(p_1L_j\) under
\(x\mapsto x+2\pi\), and by \(p_2L_j\) under
\(y\mapsto y+2\pi\). Equations \eqref{eq:class-spin-direction}
and \eqref{eq:class-selected-phase} therefore give the spin signs
\eqref{eq:class-phase-signs}. Since \(U_j,V_j\) are periodic,
\(\Psi_j\) is a global smooth section of \(E_{\phi_j}\).

\medskip
\noindent\textbf{Step 3: Verify the equations and compute the energy.}
For each \(j\), the amplitude is constant and satisfies
\(\epsilon_jA_j^2=v_ja_j\), while
\(d\alpha_j/d\ell=\kappa_j=a_jK\circ\gamma_j\).
Thus the cancellation and curvature balance in Step~3 of the
proof of \Cref{thm:nonflat-surface} give
\begin{equation}\label{eq:class-equations}
 \begin{gathered}
 \tnabla\Psi_j=0,\qquad \D^{\phi_j}\Psi_j=0,\\
 \cR_{g_0}(\phi_j,\Psi_j)
 =-v_j^2a_j(K\circ\phi_j)U_j^\phi=\tau_{g_0}(\phi_j).
 \end{gathered}
\end{equation}
Equation \eqref{eq:class-map-differential} gives
\(\rank(d\phi_j)=1\), and
\(|\tau_{g_0}(\phi_j)|\ge v_j^2|a_j|k_0>0\).
Using \eqref{eq:class-length-energy} and
\(|d\phi_j|_{g_0,h}^2=v_j^2=\rho^2v_{b_j}^2\), we obtain
\begin{equation}\label{eq:prescribed-map-energy}
 \begin{aligned}
 E_{g_0}(\phi_j)
 &=\frac{\rho^2v_{b_j}^2}{2}\Area(M,g_0)
  =\frac{L_j^2\rho^2}{8\pi^2}\Area(M,g_0)\\
 &=\frac{\rho^2\Area(M,g_0)}{2\pi}I_{\mathfrak c}(b_j)
 \longrightarrow\infty
 \end{aligned}
\end{equation}
by \Cref{lem:class-magnetic-curves} and \(b_j\to\infty\).

For \(g=e^{2u}g_0\), apply \Cref{lem:conformal} with
\(\Psi_j^g=e^{-u/2}\beta\Psi_j\), proving
\Cref{prop:compatible-homotopy}. For the trivial spin structure,
\eqref{eq:class-spin-compatibility} holds for every primitive
\(p\), which proves \Cref{thm:negative-homotopy}.
\end{proof}

Condition \eqref{eq:class-spin-compatibility} characterizes the
spin structures admitted by this construction with parallel
twisted spinors for \(g_0\).

\section{Higher-dimensional targets with a fixed metric}
\label{sec:higher}

An independent harmonic circle component raises the rank of the
solutions from \Cref{sec:surfaces} while preserving their coupling.

\begin{lemma}[Adding a harmonic circle component]
\label{lem:harmonic-circle-product}
Let \((\phi,\Psi):(M,g)\to(N,h)\) be a Dirac-harmonic map
from a Riemannian spin surface, and let
\(\zeta:(M,g)\to S^1(1)\) be harmonic. For the product metric
on \(N\times S^1(1)\), set
\[
 \Phi=(\phi,\zeta),\qquad \widehat\Psi=(\Psi,0),
\]
using the splitting
\(\Phi^*T(N\times S^1)=\phi^*TN\oplus\zeta^*TS^1\)
and the induced splitting of the twisted spinor bundle.
Then \((\Phi,\widehat\Psi)\) is Dirac-harmonic, with
\[
 \tau_g(\Phi)=(\tau_g(\phi),0),\qquad
 E_g(\Phi)=E_g(\phi)+E_g(\zeta).
\]
In particular, coupling and nowhere-vanishing tension are
preserved. If \(\Psi\) is parallel, then so is
\(\widehat\Psi\).
\end{lemma}

\begin{proof}
The product connection preserves the two summands and has no
mixed curvature terms. Since \(\widehat\Psi=(\Psi,0)\) and
\(\zeta\) is harmonic,
\[
 \begin{aligned}
 \widetilde\nabla^{\Phi}_X\widehat\Psi
   &=(\widetilde\nabla^{\phi}_X\Psi,0),\\
 \tau_g(\Phi)&=(\tau_g(\phi),0),\\
 \cR(\Phi,\widehat\Psi)&=(\cR(\phi,\Psi),0).
 \end{aligned}
\]
These identities imply the equations and the parallelness assertion.
The energy identity follows from
\(|d\Phi|^2=|d\phi|^2+|d\zeta|^2\).
\end{proof}

\begin{proof}[Proof of \Cref{cor:rank-two-product}]
Use the solutions \((\phi_m,\Psi_m)\) from
\Cref{thm:nonflat-surface}, retaining their source data
\(g=e^{2u}g_0\), \(p\), and \(t\).
Equation \eqref{eq:Cprime-map} gives
\[
 d\phi_m(X)=\frac{mL}{2\pi}\,dt(X)V_m,\qquad
 \ker d\phi_m=\ker dt.
\]
Define
\begin{equation}\label{eq:independent-circle-map}
 \zeta([x,y])=\exp\bigl(\ii(-p_2x+p_1y)\bigr)\in S^1(1).
\end{equation}
Integrality of \(p\) makes \(\zeta\) global. Its local angular
lifts are affine, so it is \(g_0\)-harmonic and hence
\(g\)-harmonic by conformal covariance. Moreover,
\[
 d\zeta=\ii\zeta(-p_2\,dx+p_1\,dy),\qquad
 dt\wedge(-p_2\,dx+p_1\,dy)
     =(p_1^2+p_2^2)\,dx\wedge dy\ne0.
\]
Thus \(\ker d\phi_m\cap\ker d\zeta=\{0\}\), and
\(\Phi_m=(\phi_m,\zeta)\) has rank two everywhere.
By \Cref{lem:harmonic-circle-product},
\((\Phi_m,(\Psi_m,0))\) is Dirac-harmonic, with nowhere-vanishing
tension and parallel twisted spinors for \(g_0\). Since \(\zeta\)
is fixed,
\[
 E_g(\Phi_m)=E_g(\phi_m)+E_g(\zeta)\longrightarrow\infty.
\]
For nonorientable \(N\), the local-isometry projection used in
\Cref{thm:nonflat-surface} preserves \(\ker d\phi_m\), so the
same argument applies.
\end{proof}

The construction applies to the other rank-one families in this
paper. The nonzero winding vector \((-p_2,p_1)\) of \(\zeta\)
makes the resulting rank-two maps non-null-homotopic.

For general higher-dimensional targets, the curve construction
requires the covariant derivatives of \(U,V\) and the vector
\(R(U,V)V\) to remain in their span. The following criterion
states the required identities along the curve.

\begin{theorem}[A curve criterion in arbitrary dimension]
\label{thm:ambient-curve}
Let \((N,h)\) contain a smooth unit-speed closed immersed curve
\(\gamma:\RR/L\ZZ\to N\), and let \(U,V=\gamma'\) be smooth
periodic orthonormal fields along \(\gamma\). Suppose
\begin{equation}\label{eq:ambient-hypotheses}
 \nabla_VU=\kappa V,\qquad
 \nabla_VV=-\kappa U,\qquad
 R^N(U,V)V=kU,\qquad \kappa=ak,
\end{equation}
where \(a\ne0\) is constant and \(k\) is nowhere zero.
Assume also
\begin{equation}\label{eq:ambient-phase}
 e^{\ii m_0\Theta_\gamma}=-1
 \quad\text{for some }m_0\in\NN,
 \qquad \Theta_\gamma=\int_0^L\kappa\,d\ell.
\end{equation}
Then every closed Riemannian spin torus admits a sequence of
coupled Dirac-harmonic maps into \(N\) satisfying
\eqref{eq:main-conclusions}, with parallel twisted spinors for a
flat conformal representative. The maps are null-homotopic if
\(\gamma\) is contractible in \(N\).
\end{theorem}

\begin{proof}[Proof of \Cref{thm:ambient-curve}]
Use the construction \eqref{eq:Cprime-map}--\eqref{eq:Cprime-spinor}
with the given curve data, \(\epsilon=\operatorname{sign}(a)\),
phase \(\alpha(\ell)=\int_0^\ell\kappa(s)\,ds\), and traversal
numbers \eqref{eq:Cprime-coverings}.
Condition \eqref{eq:ambient-phase} gives the spin signs
\eqref{eq:Cprime-descent}, so the spinors are globally defined.
With \(\kappa_m=\kappa\circ\ell_m\) and \(k_m=k\circ\ell_m\),
the first two identities in \eqref{eq:ambient-hypotheses} give
\[
 \nabla_V(U+\ii V)=-\ii\kappa(U+\ii V),\qquad
 \tnabla\Psi_m=0,\qquad
 \tau(\phi_m)=-v_m^2\kappa_m U_m.
\]
To check the ambient curvature force, extend \(U_m,V_m\) to a
local orthonormal target frame. The only nonzero target coefficients
of \(\Psi_m\) are \(\xi_m\) and \(\ii\xi_m\), and
\(E_1\cdot\xi_m=\ii\epsilon\xi_m\) gives
\[
 \Rea\langle\xi_m,E_1\cdot(\ii\xi_m)\rangle=-\epsilon A_m^2,
 \qquad
 \Rea\langle\ii\xi_m,E_1\cdot\xi_m\rangle=\epsilon A_m^2.
\]
All diagonal coefficients vanish. Substituting in
\eqref{eq:vector-curvature}, with
\(d\phi_m(E_1)=v_mV_m\) and \(d\phi_m(E_2)=0\), gives
\begin{align}
 \cR(\phi_m,\Psi_m)
 &=\frac{v_m}{2}\bigl(-\epsilon A_m^2R^N(U_m,V_m)V_m
                   +\epsilon A_m^2R^N(V_m,U_m)V_m\bigr)\notag\\
 &=-\epsilon v_mA_m^2R^N(U_m,V_m)V_m\notag\\
 &=-v_m^2ak_mU_m=-v_m^2\kappa_m U_m=\tau(\phi_m).
 \label{eq:ambient-force}
\end{align}
Rank, nonvanishing tension, and energy follow from
\eqref{eq:Cprime-differential} and \eqref{eq:Cprime-energy}.
Apply \Cref{lem:conformal} for the prescribed source metric.
If \(\gamma\) is contractible, so is
\(\phi_m=\gamma\circ\ell_m\).
\end{proof}

The vector identity \(R^N(U,V)V=kU\) excludes components
orthogonal to \(\operatorname{span}\{U,V\}\), which the sectional
curvature condition \(\langle R^N(U,V)V,U\rangle=k\) does not.
The two Frenet identities likewise exclude additional normal
components of the covariant derivatives.

\begin{proof}[Proof of \Cref{cor:nonflat-TG-surface}]
Let \(\iota:P\to N\) be the given immersion. Apply
\Cref{thm:nonflat-surface} to \(P\), then compose using
\Cref{lem:totally-geodesic}. Since \(\iota\) is isometric,
rank and energy are preserved; null-homotopy is preserved by
composition.
\end{proof}

The same argument applies to a totally geodesic surface disk
containing a curve and its filling from
\Cref{lem:short-magnetic-curves}, as illustrated by
\Cref{cor:spaceforms}.

Two classes of warped products supply totally geodesic surfaces.
Let \((P,h_P)\) be a closed nonflat surface and \((Q,h_Q)\)
a closed Riemannian manifold. In
\begin{equation}\label{eq:base-warp-product}
 N=P\times_f Q,\qquad h=h_P+f(p)^2h_Q,
 \qquad f:P\to(0,\infty),
\end{equation}
every slice \(P\times\{q\}\) is totally geodesic. For lifted
fields \(X,Y\) on \(P\) and \(Z\) on \(Q\), the Koszul
formula gives
\[
 2h(\nabla_XY,Z)=-Z(h_P(X,Y))=0,
\]
so \(\mathrm{II}=0\). The choice \(f\equiv1\) includes
ordinary products.

In the opposite order, let \((B,h_B)\) be closed and put
\begin{equation}\label{eq:fiber-warp-product}
 N=B\times_f P,\qquad h=h_B+f(b)^2h_P,
 \qquad f:B\to(0,\infty).
\end{equation}
For lifted fiber fields \(X,Y\) and a base field \(Z\),
\[
 2h(\nabla_XY,Z)=-Z\bigl(f^2h_P(X,Y)\bigr)
               =-2h(X,Y)Z(\log f).
\]
Hence \(\mathrm{II}(X,Y)=-h(X,Y)\nabla^B\log f\).
Compactness of \(B\) supplies a critical point \(b_0\) of \(f\),
where the fiber is totally geodesic with nonflat induced metric
\(f(b_0)^2h_P\). Thus both warped products give examples in
every target dimension at least three.

\section{Choosing the target metric}
\label{sec:metric-choice}

We prove \Cref{thm:metric-choice} by making the target metric
round in one coordinate ball and applying the local space-form
construction.

\begin{proof}[Proof of \Cref{thm:metric-choice}]
Choose a coordinate ball \(B\) with
\(\overline B\subset\mathcal O\), and concentric smaller balls
\(B_1\Subset B_2\Subset B\). Rescale the chart so that \(B\)
has coordinate radius less than \(\pi/2\), and put
\begin{equation}\label{eq:implanted-round-metric}
 h_{\mathrm{rd}}=dr^2+\sin^2r\,g_{S^{n-1}}.
\end{equation}
This metric is smooth at the origin, being the round metric
in normal coordinates.

Choose \(\varpi\in C_c^\infty(B_2)\) with
\(0\le\varpi\le1\) and \(\varpi=1\) near \(\overline B_1\)
\cite[Proposition~2.25]{Lee13}. Extend it by zero and define
\begin{equation}\label{eq:patched-target-metric}
 h=
 \begin{cases}
 \varpi h_{\mathrm{rd}}+(1-\varpi)h_{\mathrm{ref}},&\text{on }B,\\
 h_{\mathrm{ref}},&\text{on }N\setminus B.
 \end{cases}
\end{equation}
The convex combination is positive definite and agrees smoothly
with \(h_{\mathrm{ref}}\) near \(\partial B\). Thus \(h\) is
a smooth metric, equal to \(h_{\mathrm{rd}}\) near
\(\overline B_1\) and to \(h_{\mathrm{ref}}\) outside
\(B_2\subset\mathcal O\).

The local construction in the proof of \Cref{cor:spaceforms}
supplies a totally geodesic warped annulus and the required
solutions inside \(B_1\). They are null-homotopic since \(B_1\)
is contractible. The metric and annulus are independent of the
source metric and spin structure.
\end{proof}

\end{document}